\documentclass[12pt,draftcls,onecolumn]{IEEEtran}  

\IEEEoverridecommandlockouts                              %

\def\calD{{\mathcal{D}}}

\def\calL{{\mathcal{L}}}

\def\calP{{\mathcal{P}}}

\def\calU{{\mathcal{U}}}

\def\calX{{\mathcal{X}}}

\def\Real{{\mathbb{R}}}

\let\labelindent\relax
\usepackage{epsfig} 
\usepackage{epstopdf}

\usepackage{fancyhdr}
\usepackage{indentfirst}
\usepackage{amsmath}
\usepackage{amssymb}
\usepackage{algorithm}
\usepackage{algorithmic}
\usepackage{enumitem}

\usepackage{amsthm}

\usepackage{graphicx}
\usepackage{color}

\definecolor{gray}{RGB}{128,128,128}

\usepackage[short]{optidef}
\usepackage{latexsym}
\usepackage{amsfonts}
\usepackage{amssymb,amsmath,amsfonts}
\usepackage{amsmath,mathrsfs,bm,url,times}
\usepackage[table]{xcolor}
\usepackage{cite}
\usepackage{caption}
\usepackage{subcaption}
\newtheorem{assumption}{Assumption}

\newtheorem{theorem}{Theorem}
\newtheorem{lemma}{Lemma}

\DeclareMathOperator{\argmax}{argmax}

\DeclareMathOperator*{\argmin}{arg\,min}

\DeclareMathOperator{\NetReg}{Net-Reg}
\DeclareMathOperator{\col}{col}
\DeclareMathOperator{\inout}{in}
\DeclareMathOperator{\outin}{out}
\usepackage{array}
\newcolumntype{M}[1]{>{\centering\arraybackslash}m{#1}}
\newcolumntype{N}{@{}m{0pt}@{}}

\usepackage{multirow}

\allowdisplaybreaks
\title{\LARGE \bf Distributed Online Convex Optimization with Adversarial Constraints: Reduced Cumulative Constraint Violation Bounds under Slater's Condition
}

\author{Xinlei Yi, Xiuxian Li, Tao Yang, Lihua Xie,\\ Yiguang Hong, Tianyou Chai, and Karl H. Johansson
\thanks{X. Yi is with Lab for Information \& Decision Systems, Massachusetts Institute of Technology, Cambridge, MA 02139, USA.
        {\tt\small xinleiyi@mit.edu}.}%
\thanks{X. Li and Y. Hong are with Department of Control Science and Engineering, College of Electronics and Information Engineering, Tongji University, Shanghai, 201804, China,  and they are also affiliated with Shanghai Research Institute for Intelligent Autonomous Systems. {\tt\small xli@tongji.edu.cn}, {\tt\small yghong@iss.ac.cn}.}
\thanks{T. Yang and T. Chai are with the State Key Laboratory of Synthetical Automation for Process Industries, Northeastern University, 110819, Shenyang, China. {\tt\small \{yangtao,tychai\}@mail.neu.edu.cn}.}
\thanks{L. Xie is with School of Electrical and Electronic Engineering,
Nanyang Technological University, 50 Nanyang Avenue, Singapore 639798. {\tt\small elhxie@ntu.edu.sg}.}
\thanks{K. H. Johansson is with Division of Decision and Control Systems, School of Electrical Engineering and Computer Science, KTH Royal Institute of Technology, and he is also affiliated with Digital Futures, 10044, Stockholm, Sweden. {\tt\small  kallej@kth.se}.}

}

\begin{document}

\maketitle
\thispagestyle{plain}
\pagestyle{plain}

\begin{abstract}\label{online_op:Abstract}
This paper considers distributed online convex optimization with adversarial constraints. In this setting, a network of agents makes decisions at each round, and then only a portion of the loss function and a coordinate block of the constraint function are privately revealed to each agent. The loss and constraint functions are convex and can vary arbitrarily across rounds. The agents collaborate to minimize network regret and cumulative constraint violation. A novel distributed online algorithm is proposed and it achieves an $\mathcal{O}(T^{\max\{c,1-c\}})$ network regret bound and an $\mathcal{O}(T^{1-c/2})$ network cumulative constraint violation bound, where $T$ is the number of rounds and $c\in(0,1)$ is a user-defined trade-off parameter. When Slater's condition holds (i.e, there is a point that strictly satisfies the inequality constraints), the network cumulative constraint violation bound is reduced to $\mathcal{O}(T^{1-c})$. Moreover, if the loss functions are strongly convex, then the network regret bound is reduced to $\mathcal{O}(\log(T))$, and the network cumulative constraint violation bound is reduced to $\mathcal{O}(\sqrt{\log(T)T})$ and $\mathcal{O}(\log(T))$ without and with Slater's condition, respectively. To the best of our knowledge, this paper is the first to achieve reduced (network) cumulative constraint violation bounds for (distributed) online convex optimization with adversarial constraints under Slater's condition. Finally, the theoretical results are verified through numerical simulations.

\end{abstract}
\section{INTRODUCTION}\label{online_opsec:intro}
Online convex optimization is a sequential decision making problem with a sequence of arbitrarily varying convex loss functions. Specifically, at each round $t$, a decision maker selects a decision $x_t\in\calX$, where $\calX\subseteq\mathbb{R}^p$ is a known closed convex set with $p$ being a positive integer. After the selection, a convex loss  function $l_t:\mathbb{R}^p\rightarrow \mathbb{R}$ is  revealed. The goal of the decision maker is to minimize the cumulative loss across $T$ rounds. The standard performance measure is regret
\begin{align*}
\sum_{t=1}^{T}l_t(x_t)-\min_{x\in\mathcal{X}}\sum_{t=1}^{T}l_t(x),
\end{align*}
which is the difference between the accumulative loss and the loss obtained by the best fixed decision in hindsight.
Due to its wide applications such as online display advertising \cite{goldfarb2011online}, online linear regression  \cite{wang2014online}, and reactive power management \cite{kekatos2014stochastic}, online convex optimization has been extensively studied over the past decades, e.g.,
\cite{zinkevich2003online,hazan2007logarithmic,agarwal2010optimal,zhang2018dynamic,pmlr-v124-zhang20a,ijcai2020-731}. For more applications and  background, please refer to the monographs \cite{shalev2012online,hazan2016introduction}.

\subsection{Online Convex Optimization with Adversarial Constraints}
It is well-known that the projection-based online gradient descent algorithm
\begin{align}\label{online_op:intro_ogd_alg}
	x_{t+1}=\calP_{\mathcal{X}}(x_t-\alpha \nabla l_t(x_t)),
\end{align}
where $\calP_{\mathcal{X}}(\cdot)$ is the projection onto  set $\mathcal{X}$ and $\alpha>0$ is the stepsize, achieves an $\mathcal{O}(\sqrt{T})$ regret bound for convex loss functions with bounded subgradients \cite{zinkevich2003online}, which is a tight bound up to constant factors \cite{hazan2007logarithmic}. The regret bound can be reduced under more stringent strong convexity conditions on the loss functions \cite{hazan2007logarithmic,shalev2012online,hazan2016introduction}. Although the algorithm \eqref{online_op:intro_ogd_alg} seems simple, the projection operator can yield heavy computation and/or storage burden when the constraint set is complicated. For example, in practice, the constraint set $\mathcal{X}$ is often characterized by inequality constraints, i.e.,
\begin{align}\label{online_op:intro_const_set}
\mathcal{X}=\{x:~g(x)\le\bm{0}_{m},~x\in\mathbb{X}\},
\end{align}
where $m$ is a positive integer, $\mathbb{X}\subseteq\mathbb{R}^{p}$ is a closed convex set which normally is a simple set, e.g., a box or a ball, and $g(x):\mathbb{R}^{p}\rightarrow \mathbb{R}^m$ is the constrain function which is convex. To tackle this challenge, online convex optimization with long-term constraints has been considered, e.g., \cite{mahdavi2012trading,jenatton2016adaptive,yu2020lowJMLR}. In this new problem, the decisions are selected from the simple set $\mathbb{X}$ and the inequality constraints  should be satisfied in the long-term on average. This is measured by constraint violation
\begin{align}\label{online_op:intro_def_cons}
	\Big\|\Big[\sum_{t=1}^Tg(x_{t})\Big]_+\Big\|,
\end{align}
which is the violation of the accumulative constraint function. Here $[\:\cdot\:]_+$ is the projection onto the nonnegative space. In this case, the goal of the decision maker is to minimize both regret and constraint violation.

The problem above has been further extended to the adversarial constraints setting, i.e., the constraint function is time-varying, can be arbitrarily and adversarially designed, and is revealed to the decision maker after selecting its decision at each round. In this case, let $g_t:\mathbb{R}^{p}\rightarrow \mathbb{R}^m$ denote the constraint function of the $t$-th round and then the constraint violation metric is $\|[\sum_{t=1}^Tg_t(x_{t})]_+\|$. Online convex optimization with adversarial constraints has also been extensively studied, e.g.,  \cite{paternain2016online,sun2017safety,yu2017online,neely2017online,chen2017online,pmlr-v97-liakopoulos19a,sadeghi2019online,wei2020online}. These works usually propose online primal--dual algorithms and achieve sublinear regret and constraint violation bounds. For example,  \cite{sun2017safety} achieved an $\mathcal{O}(\sqrt{T})$ regret bound and an $\mathcal{O}(T^{3/4})$ constraint violation bound. In \cite{yu2017online,neely2017online} the bound for constraint violation was reduced to $\mathcal{O}(\sqrt{T})$ under Slater's condition (i.e., there exists a point $x_s\in \mathbb{X}$ and a constant $\epsilon_s>0$, such that $g_{t}(x_s)\le-\epsilon_s{\bf 1}_{m}$ for all $t$). Slater's condition is an example of a constraint qualification that guarantees strong duality in convex optimization problems with inequality constraints \cite{boyd2004convex}. This condition guarantees that the optimal solution of the primal problem can be achieved without violating any of the constraints. In the context of online convex optimization, under Slater's condition, algorithms can adaptively adjust the decision at each round to ensure that the constraints are not violated significantly, even if the optimization problem is evolving dynamically.

Note that the constraint violation metric defined in \eqref{online_op:intro_def_cons} takes the summation across rounds before the projection operation $[\:\cdot\:]_+$. As a result, it allows strict feasible decisions that have large margins compensate constraint violations at many rounds. In this way, even if the constraint violation grows sublinearly, the constraints could be violated at many rounds. To avoid this potential drawback, stricter forms of constraint violation metric have been proposed in \cite{NIPS2018_7852,yi2021regret}. For instance, \cite{NIPS2018_7852} proposed cumulative constraint violation
\begin{align}\label{online_op:regc1}
	\Big\|\sum_{t=1}^T[g(x_{t})]_+\Big\|,
\end{align}
and cumulative squared constraint violation
\begin{align}\label{online_op:regc2}
	\sum_{t=1}^T\|[g(x_{t})]_+\|^2.
\end{align}
Both forms of metric \eqref{online_op:regc1} and \eqref{online_op:regc2} take into account all constraints that are not satisfied, and the metric \eqref{online_op:regc1} is stricter than the constraint violation metric defined in \eqref{online_op:intro_def_cons}.
In  \cite{NIPS2018_7852}, an $\mathcal{O}(T^{\max\{c,1-c\}})$ regret bound and an $\mathcal{O}(T^{1-c/2})$  cumulative constraint violation bound have been achieved, where $c\in(0,1)$ is a user-defined trade-off parameter enabling the trade-off between these two bounds. Moreover, when the loss functions are strongly convex, the regret and  cumulative constraint violation bounds can be reduced to $\mathcal{O}(\log(T))$ and $\mathcal{O}(\sqrt{\log(T)T})$, respectively. The key idea to achieve these results is to use the clipped constraint function $[g]_+$ to replace the original constraint function $g$. As pointed out in \cite{NIPS2018_7852}, with this idea, the bounds for constraint violation achieved in some existing works, e.g., \cite{mahdavi2012trading,jenatton2016adaptive}, still hold when using the stricter metric \eqref{online_op:regc1} to replace the standard metric \eqref{online_op:intro_def_cons}. However, this idea becomes ineffective when extending the constraint violation bounds achieved in \cite{yu2017online,neely2017online}, as the use of the clipping operation renders Slater's condition ineffective\footnote{Noting that the clipped function is nonnegative, it is impossible to find a point where the value of the clipped function is strictly less than zero.}. It remains an open problem {\it how to achieve reduced cumulative constraint violation bounds for online convex optimization with adversarial constraints under Slater's condition}.

\subsection{Distributed Online Convex Optimization with Adversarial Constraints}
Noting that distributed paradigm can address critical issues in centralized processing such as data privacy, data security, and single point failures, distributed online convex optimization has also been extensively studied, e.g., \cite{mateos2014distributed,koppel2015saddle,hosseini2016online,zhang2017projection,shahrampour2018distributed,akbari2018individual,wan2020projection,yuan2020distributed,carnevale2020distributed,lee2016coordinate,lee2017stochastic}. In distributed online convex optimization, there are a group of agents (decision makers). At each round $t$, each agent $i$ selects a decision $x_{i,t}\in\mathcal{X}$, and after the selection a portion of the global loss function $l_{t}$ is revealed to agent $i$ only. The goal of the agents is to minimize the network-wide accumulated loss, and the corresponding performance measure is network regret
\begin{align*}
\frac{1}{n}\sum_{i=1}^{n}\Big(\sum_{t=1}^{T}l_{t}(x_{i,t})
	-\min_{x\in\mathcal{X}}\sum_{t=1}^{T}l_{t}(x)\Big).
\end{align*}

Similarly, in order to avoid the potential computation and/or storage challenge caused by the projection operator when using projection-based algorithms, distributed online convex optimization with long-term constraints has also been considered, e.g., \cite{yuan2017adaptive,yuan2021distributed,yuan2021distributedb}. In this problem, the constraint set $\mathcal{X}$ is characterized by inequality constraints as described in \eqref{online_op:intro_const_set}, and each agent knows the simple set $\mathbb{X}$ and the constraint function $g$ in advance. Similar to the centralized case, the decisions are selected from $\mathbb{X}$ instead of $\mathcal{X}$ and the inequality constraints should be satisfied in the long-term on average, which is measured by network constraint violation
\begin{align*}
	\frac{1}{n}\sum_{i=1}^n\Big\|\Big[\sum_{t=1}^Tg(x_{i,t})\Big]_+\Big\|
\end{align*}
or network cumulative constraint violation
\begin{align*}
	\frac{1}{n}\sum_{i=1}^n\sum_{t=1}^T\|[g(x_{i,t})]_+\|.
\end{align*}
The network regret and  constraint violation bounds achieved in \cite{yuan2017adaptive,yuan2021distributed,yuan2021distributedb} are similar to their centralized counterparts. Specifically, \cite{yuan2017adaptive} achieved an $\mathcal{O}(T^{0.5+\beta})$ network regret bound and an $\mathcal{O}(T^{1-\beta/2})$ network constraint violation bound, where $\beta\in(0,0.5)$ is a user-defined parameter. When the loss functions are quadratic and the constraint function is linear,  \cite{yuan2021distributed} achieved an $\mathcal{O}(T^{\max\{c,1-c\}})$ network regret bound and an $\mathcal{O}(T^{1-c/2})$ network cumulative constraint violation bound. The quadratic loss and linear constraint functions were relaxed by convex functions in  \cite{yuan2021distributedb} and the same  network regret and cumulative constraint violation bounds were still achieved.

The distributed online convex optimization with long-term constraints setting was extended to a more general scenario in \cite{yi2021regretTAC}, where the constraint function is adversarial and at each round only a coordinate block of the global constraint function is privately revealed to each agent after selecting its decision. The same network regret and cumulative constraint violation bounds as achieved in \cite{yuan2021distributed,yuan2021distributedb} were also established.  However, similar to the centralized case, it is still unclear {\it how to reduce network cumulative constraint violation bounds for distributed online convex optimization with adversarial constraints under Slater's condition}, which is the main motivation behind this paper.

\subsection{Main Contributions}

In this paper, similar to \cite{yi2021regretTAC}, we study the general distributed online convex optimization with adversarial constraints, and adopt network regret and cumulative constraint violation as performance measures. However, different from \cite{yi2021regretTAC}, we also consider the scenario where Slater's condition holds. To tackle this problem, we propose a novel distributed online primal--dual composite mirror descent algorithm. To avoid the ineffectiveness of Slater's condition, the proposed algorithm does not simply replace the original constraint function with the clipped constraint function, as done in \cite{NIPS2018_7852,yuan2021distributed,yuan2021distributedb,yi2021regretTAC}. Instead, it updates the dual variables by directly maximizing the regularized Lagrangian function, which can be explicitly calculated using the clipped constraint function. Consequently, for the scenario without Slater's condition, all squared constraint violations can be accumulated, leveraging the fact that $a^\top[a]_+=\|[a]_+\|^2$ for any vector $a$, and state-of-the-art (network) cumulative constraint violation bounds can be achieved. Moreover, for the scenario with Slater's condition, all constraint violations can be accumulated directly by summing the dual variables. As a result, we can demonstrate that reduced (network) cumulative constraint violation bounds can be achieved under Slater's condition through a properly designed analysis.

The performance guarantees for the proposed algorithm are summarized below.
\begin{itemize}
	\item We show in Theorem~\ref{online_op:corollaryreg} that the proposed algorithm achieves an $\mathcal{O}(T^{\max\{c,1-c\}})$ network regret bound and an $\mathcal{O}(T^{1-c/2})$ network cumulative constraint violation bound as achieved in \cite{yi2021regretTAC}, which generalizes the results in \cite{NIPS2018_7852,sun2017safety,yuan2021distributed,yuan2021distributedb} to more general settings,  and also improves the results in \cite{yuan2017adaptive}.

	\item When Slater's condition holds, we show in Theorem~\ref{online_op:corollaryreg_slater} that network cumulative constraint violation bound is reduced to $\mathcal{O}(T^{1-c})$, which generalizes the results in \cite{yu2017online,neely2017online}, and thus solves the open problem left in \cite{NIPS2018_7852}. To the best of our knowledge, this paper is the first to show a reduced (network) cumulative constraint violation bound can be achieved for (distributed) online convex optimization with adversarial constraints under Slater's condition.

	\item When the loss functions are strongly convex, we show in Theorem~\ref{online_op:corollaryreg_scmu} that the proposed algorithm achieves an $\mathcal{O}(\log(T))$ network regret bound and an $\mathcal{O}(\sqrt{\log(T)T})$ network cumulative constraint violation bound, which generalizes the results in \cite{NIPS2018_7852,yuan2021distributedb} and improves the results in \cite{yuan2017adaptive,yi2021regretTAC}. Moreover, if Slater's condition holds in addition, network cumulative constraint violation bound is reduced to $\mathcal{O}(\log(T))$. Again, to the best of our knowledge, it is the first time to achieve such a result.
\end{itemize}

The detailed comparison of this paper to related works is summarized in TABLE~\ref{online_op::table}.

\bgroup
\def\arraystretch{1.15}
\begin{table*}[!]
\caption{Comparison of this paper to related works on online convex optimization with long-term and adversarial constraints.}
\label{online_op::table}
\centering
\small
\begin{tabular}{M{1.1cm}|M{1.5cm}|M{1.8cm}|M{1.8cm}|M{1.2cm}|M{2.2cm}|M{1.6cm}|M{2.0cm}N}
\hline
Reference&Problem type&Loss functions&Constraint functions&Slater's condition&Regret&Constraint violation&Cumulative constraint violation&\\

\hline
\cite{sun2017safety}&Centralized&Convex&Convex&No&$\mathcal{O}(\sqrt{T})$& $\mathcal{O}(T^{3/4})$&Not given&\\

\hline
\cite{yu2017online,neely2017online}&Centralized&Convex
&Convex&Yes&$\mathcal{O}(\sqrt{T})$& $\mathcal{O}(\sqrt{T})$&Not given&\\

\hline
\multirow{2}{*}[-0.8em]{\cite{NIPS2018_7852}}&\multirow{2}{*}[-0.8em]{Centralized}&Convex
&\multirow{2}{*}[-0.8em]{\parbox{1.8cm}{\centering Convex, time-invariant}}&\multirow{2}{*}[-0.8em]{No}&$\mathcal{O}(T^{\max\{c,1-c\}})$& \multicolumn{2}{c}{$\mathcal{O}(T^{1-c/2})$}&\\
\cline{3-3}\cline{6-8}
&&Strongly convex&&&$\mathcal{O}(\log(T))$&\multicolumn{2}{c}{$\mathcal{O}(\sqrt{\log(T)T})$}&\\

\hline
\multirow{2}{*}[-0.8em]{\cite{yuan2017adaptive}}&\multirow{2}{*}[-0.8em]{Distributed}
&Convex&\multirow{2}{*}[-0.8em]{\parbox{1.8cm}{\centering Convex, time-invariant}}&\multirow{2}{*}[-0.8em]{No}&$\mathcal{O}(T^{\max\{0.5+\beta\}})$& $\mathcal{O}(T^{1-\beta/2})$&\multirow{2}{*}[-0.8em]{Not given}&\\
\cline{3-3}\cline{6-7}
&&Strongly convex&&&$\mathcal{O}(T^{c})$& $\mathcal{O}(T^{1-c/2})$&&\\

\hline
\cite{yuan2021distributed}&Distributed&Quadratic
&Linear, time-invariant&No&$\mathcal{O}(T^{\max\{c,1-c\}})$& \multicolumn{2}{c}{$\mathcal{O}(T^{1-c/2})$}&\\

\hline
\multirow{2}{*}[-0.8em]{\cite{yuan2021distributedb}}&\multirow{2}{*}[-0.8em]{Distributed}&Convex
&\multirow{2}{*}[-0.8em]{\parbox{1.8cm}{\centering Convex, time-invariant}}&\multirow{2}{*}[-0.8em]{No}&$\mathcal{O}(T^{\max\{c,1-c\}})$& \multicolumn{2}{c}{$\mathcal{O}(T^{1-c/2})$}&\\
\cline{3-3}\cline{6-8}
&&Strongly convex&&&$\mathcal{O}(\log(T))$&\multicolumn{2}{c}{$\mathcal{O}(\sqrt{\log(T)T})$}&\\

\hline
\multirow{2}{*}[-0.8em]{\cite{yi2021regretTAC}}&\multirow{2}{*}[-0.8em]{Distributed}&Convex
&\multirow{2}{*}[-0.8em]{Convex}&\multirow{2}{*}[-0.8em]{No}&$\mathcal{O}(T^{\max\{c,1-c\}})$& \multicolumn{2}{c}{\multirow{2}{*}[-0.8em]{$\mathcal{O}(T^{1-c/2})$}}&\\
\cline{3-3}\cline{6-6}
&&Strongly convex&&&$\mathcal{O}(T^c)$&\multicolumn{2}{c}{}&\\

\hline
\multirow{4}{*}{\parbox{1.1cm}{\centering This paper}}&\multirow{4}{*}{Distributed}&\multirow{2}{*}{Convex}
&\multirow{4}{*}{Convex}&No&\multirow{2}{*}{$\mathcal{O}(T^{\max\{c,1-c\}})$}& \multicolumn{2}{c}{$\mathcal{O}(T^{1-c/2})$}&\\
\cline{5-5}\cline{7-8}
&&&&Yes&& \multicolumn{2}{c}{$\mathcal{O}(T^{1-c})$}&\\
\cline{3-3}\cline{5-8}
&&\multirow{2}{*}{\parbox{1.5cm}{\centering Strongly convex}}&&No&\multirow{2}{*}{$\mathcal{O}(\log(T))$}& \multicolumn{2}{c}{$\mathcal{O}(\sqrt{\log(T)T})$}&\\
\cline{5-5}\cline{7-8}
&&&&Yes&& \multicolumn{2}{c}{$\mathcal{O}(\log(T))$}&\\
\hline
\end{tabular}
\vskip -0.1in
\end{table*}
\egroup

\noindent {\bf Outline}: The rest of this paper is organized as follows. Section~\ref{online_opsec:problem} formulates the considered problem. Section~\ref{online_opsec:algorithm} proposes the novel distributed online primal--dual composite mirror descent algorithm to solve the problem.  Section~\ref{online_opsec:main} analyzes the network regret and cumulative constraint violation bounds for the proposed algorithm. Section~\ref{online_opsec:simulation} gives numerical simulations. Finally, Section~\ref{online_opsec:conclusion} concludes the paper and proofs are given in Appendix.

\noindent {\bf Notations}: All inequalities and equalities throughout this paper are understood componentwise. $\mathbb{R}^n$ and $\mathbb{R}^n_+$ stand for the set of $n$-dimensional vectors and nonnegative vectors, respectively. $\mathbb{N}_+$ denotes the set of all positive integers. $[n]$ represents the set $\{1,\dots,n\}$ for any $n\in\mathbb{N}_+$. $\|\cdot\|$ ($\|\cdot\|_1$) stands for the Euclidean norm (1-norm) for vectors and the induced 2-norm (1-norm) for matrices. $x^\top$ denotes the transpose of a vector or a matrix. $\langle x,y\rangle$ represents the standard inner product of two vectors $x$ and $y$. ${\bf 0}_n$ (${\bf 1}_n$) denotes the column zero (one)
vector with dimension $n$. $\col(z_1,\dots,z_k)$ is the concatenated column vector of $z_i\in\mathbb{R}^{n_i},~i\in[k]$. For a closed convex set $\mathbb{K}\subseteq\mathbb{R}^p$ and any $x\in\Real^{p}$, $\calP_{\mathbb{K}}(x)$ is the projection of $x$ onto $\mathbb{K}$, i.e., $\calP_{\mathbb{K}}(x)=\argmin_{y\in\mathbb{K}}\|x-y\|^2$. For simplicity, $[x]_+$ is used to denote $\calP_{\mathbb{R}^p_+}(x)$. For a function $f$, let $\nabla  f(x)$ denote the (sub)gradient of $f$ at $x$. $\calU(a,b)$ is the uniform distribution over the interval $[a,b]$ with $a\le b\in\mathbb{R}$.

\section{Problem Formulation}\label{online_opsec:problem}
This paper studies distributed online convex optimization with adversarial constraints. Specifically, consider a network of $n$ agents indexed by $i\in[n]$. At each round $t$,  each agent $i$ makes a decision $x_{i,t}\in \mathbb{X}$, where $\mathbb{X}\subseteq\mathbb{R}^p$ is a known set and $p$ is a positive integer. After making the selection, the local loss function $l_{i,t}:\mathbb{R}^p\rightarrow \mathbb{R}$ and constraint function $g_{i,t}:\mathbb{R}^p\rightarrow \mathbb{R}^{m_i}$ are revealed to agent $i$ only, which respectively are a portion of the global loss function $l_t(x)=\frac{1}{n}\sum_{i=1}^nl_{i,t}(x)$ and a coordinate block of the global constraint function $g_t(x)=\col(g_{1,t}(x),\dots,g_{n,t}(x))$. Here, $m_i$ is a positive integer. The agents collaborate to select the decision sequences $\{x_{i,t}\}$ such that both network regret
\begin{align}\label{online_op:reg}
	\NetReg(T):=\frac{1}{n}\sum_{i=1}^{n}\Big(\sum_{t=1}^{T}l_t(x_{i,t})
	-\min_{x\in\mathcal{X}_{T}}\sum_{t=1}^{T}l_t(x)\Big)
\end{align}
and network cumulative constraint violation
\begin{align}\label{online_op:regc}
	\frac{1}{n}\sum_{i=1}^n\sum_{t=1}^T\|[g_{t}(x_{i,t})]_+\|
\end{align}
grow sublinearly, where $T\in\mathbb{N}_+$ is the number of rounds and
\begin{align*}
\mathcal{X}_{T}=\{x:~x\in \mathbb{X},~g_{t}(x)\le{\bf0}_{m},~
\forall t\in[T]\}
\end{align*} is the feasible set with $m=\sum_{i=1}^n m_i$. Similar to existing literature considering adversarial constraints, e.g., \cite{paternain2016online,sun2017safety,yu2017online,neely2017online,chen2017online,pmlr-v97-liakopoulos19a,sadeghi2019online,wei2020online,yi2021regretTAC}, it is assumed that the feasible set $\mathcal{X}_{T}$ is nonempty for every $T$.

Without loss of generality, we assume that each local loss function $l_{i,t}$ consists of a private part $f_{i,t}$ and a common part $r_{t}$, i.e., $l_{i,t}(x)=f_{i,t}(x)+r_{t}(x)$. Here, $r_{t}$ represents the common knowledge in the network. For example, $r_{t}$ could be the regularization used to influence the structure of the decisions. The above distributed problem setting incorporates various problems studied in the literature. For instance, when $g_{i,t}\equiv\bm{0}_{m_i},~\forall i\in[n],~t\in\mathbb{N}_+$, the above distributed problem becomes the problem studied in \cite{yuan2020distributed}; when  $g_{i,t}\equiv\bm{0}_{m_i}$ and $r_t\equiv0,~\forall i\in[n],~t\in\mathbb{N}_+$, the above distributed problem becomes the problem studied in various existing works, e.g., \cite{mateos2014distributed,koppel2015saddle,hosseini2016online,zhang2017projection,shahrampour2018distributed,akbari2018individual,wan2020projection,carnevale2020distributed,yuan2021distributed}; when $g_{i,t}\equiv g$ and $r_t\equiv0,~\forall i\in[n],t\in\mathbb{N}_+$ with $g$ being a known and pre-defined constraint function, the above distributed problem becomes the problem studied in \cite{yuan2017adaptive,yuan2021distributed,yuan2021distributedb}; and when $r_t\equiv0,~\forall t\in\mathbb{N}_+$, the above distributed problem becomes the problem studied in \cite{yi2021regretTAC}.

Some necessary definitions and assumptions are listed in the following.

\subsection{Graph Theory}
In this paper, the communication topology for the network of agents is  modeled by a time-varying directed graph. Specifically, let $\mathcal{G}_t=(\mathcal{V},\mathcal{E}_t)$ denote the directed graph at the $t$-th round, where $\mathcal{V}=[n]$ is the agent set and $\mathcal{E}_t\subseteq\mathcal{V}\times\mathcal{V}$ the edge set. A directed edge $(j,i)\in\mathcal{E}_t$ means that agent $i$ can receive data from agent $j$ at the $t$-th round. Let $\mathcal{N}^{\inout}_i(\mathcal{G}_t)=\{j\in [n]\mid (j,i)\in\mathcal{E}_t\}$ and $\mathcal{N}^{\outin}_i(\mathcal{G}_t)=\{j\in [n]\mid (i,j)\in\mathcal{E}_t\}$ be the sets of in- and out-neighbors, respectively, of agent $i$ at the $t$-th round. A directed path is a sequence of consecutive directed edges. A directed graph is strongly connected if there is at least one directed path
from any agent to any other agent. The associated adjacency (mixing) matrix $W_t\in\mathbb{R}^{n\times n}$  fulfills $[W_t]_{ij}>0$ if $(j,i)\in\mathcal{E}_t$ or $i=j$, and $[W_t]_{ij}=0$ otherwise.

\subsection{Bregman Divergence}

This paper uses the Bregman divergence \cite{bregman1967relaxation} to measure the distance of two points $x,~y\in\mathbb{X}$, which is defined as
\begin{align*}
	\calD_\psi(x,y)=\psi(x)-\psi(y)-\langle \nabla \psi(y),x-y\rangle,
\end{align*}
where $\psi:\mathbb{R}^p\rightarrow\mathbb{R}$ is a function which is strongly convex with convexity parameter $\sigma>0$ on the set $\mathbb{X}$, i.e.,
\begin{align*}
	\psi(x)\ge\psi(y)+\langle \nabla \psi(y),x-y\rangle+\frac{\sigma}{2}\|x-y\|^2.
\end{align*}
Thus,
\begin{align}\label{online_op:eqbergman}
	\calD_\psi(x,y)\ge\frac{\sigma}{2}\|x-y\|^2,~\forall x,y\in\mathbb{X}.
\end{align}
Moreover, $\calD_\psi(\:\cdot\:,y)$ is a strongly convex function with convexity parameter $\sigma$ for all fixed $y\in\mathbb{X}$.

Two well-known examples of Bregman divergence are the Euclidean distance $\calD_\psi(x,y)=\|x-y\|^2$ generated from $\psi(x)=\|x\|^2$, and the Kullback--Leibler (K--L) divergence $\calD_\psi(x,y)=-\sum_{i=1}^px_i\log(y_i/x_i)$ between two $p$-dimensional standard unit vectors (with $\mathbb{X}$ being the $p$-dimensional probability simplex) generated from $\psi(x)=\sum_{i=1}^px_i\log x_i-x_i$.

\subsection{Assumptions}
We make the following assumptions on the loss and constraint functions.
\begin{assumption}\label{online_op:assfunction}
	The set $\mathbb{X}$ is closed and convex. For all $i\in[n]$ and $t\in\mathbb{N}_+$, the functions $r_{t}$, $f_{i,t}$, and $g_{i,t}$ are convex.
\end{assumption}

\begin{assumption}\label{online_op:ass_ftgtupper}
There exists a positive constant $F$ such that for all $i\in[n]$, $t\in\mathbb{N}_+$, and $x,y\in\mathbb{X}$,
\begin{align}
|l_{i,t}(x)-l_{i,t}(y)|\le F.\label{online_op:ftgtupper}
\end{align}
\end{assumption}

\begin{assumption}\label{online_op:ass_subgradient}
For all $i\in[n]$, $t\in\mathbb{N}_+$, and $x\in\mathbb{X}$, the subgradients $\nabla r_{t}(x)$, $\nabla f_{i,t}(x)$, and $\nabla g_{i,t}(x)$ exist, moreover, there exist positive constants $G_1$ and $G_2$ such that
\begin{align}\label{online_op:subgupper}
\|\nabla l_{i,t}(x)\|\le G_1,~\|\nabla g_{i,t}(x)\|\le G_2.
\end{align}
\end{assumption}

Note that we do not assume that the local constraint functions $\{g_{i,t}\}$ are uniformly bounded, and the assumption that $\nabla l_{i,t}(x)$ is bounded is slightly weaker than the assumption that both $\nabla f_{i,t}(x)$ and $\nabla r_{t}(x)$ are bounded.
From Assumptions~\ref{online_op:assfunction} and \ref{online_op:ass_subgradient}, and Lemma~2.6 in \cite{shalev2012online}, it follows that for all $i\in[n],~t\in\mathbb{N}_+,~x,~y\in \mathbb{X}$,
\begin{subequations}
\begin{align}
&\left|l_{i,t}(x)-l_{i,t}(y)\right|\le G_1\|x-y\|,\label{online_op:assfunction:functionLipf}\\
&\left\|g_{i,t}(x)-g_{i,t}(y)\right\|\le G_2\|x-y\|.\label{online_op:assfunction:functionLipg}
\end{align}
\end{subequations}

The following commonly used assumption is made on the graph.
\begin{assumption}\label{online_op:assgraph}
	For any $t\in\mathbb{N}_+$, the directed graph $\mathcal{G}_t$ satisfies the following conditions:
	\begin{enumerate}[label=(\alph*)]
		\item There exists a constant $w\in(0,1)$, such that $[W_t]_{ij}\ge w$ if $[W_t]_{ij}>0$.
		\item The mixing matrix $W_t$ is doubly stochastic, i.e., $\sum_{i=1}^n[W_t]_{ij}=\sum_{j=1}^n[W_t]_{ij}=1,~\forall i,j\in[n]$.
		\item There exists an integer $B>0$ such that the directed graph $(\mathcal{V},\cup_{l=0}^{ B -1}\mathcal{E}_{t+l})$ is strongly connected.
	\end{enumerate}
\end{assumption}

Some assumptions on the Bergman divergence are stated as follows.

\begin{assumption}\label{online_op:assbregman}
	For any $x\in\mathbb{X}$, $\calD_\psi(x,\:\cdot\:):\mathbb{X}\rightarrow\mathbb{R}$ is convex.
\end{assumption}

\begin{assumption}\label{online_op:assbregmanlip}
	There exists a positive constant $K$ such that
\begin{align}\label{online_op:bregmanupp}
	\calD_\psi(x,y)\le K,~\forall x,y\in\mathbb{X}.
\end{align}
\end{assumption}
Assumption~\ref{online_op:assbregman} is satisfied for commonly used Bregman divergences,such as the Euclidean distance and the K--L divergence.
Assumption~\ref{online_op:assbregmanlip} is essentially employed to ensure that the set $\mathbb{X}$ has a bounded diameter. However, in certain scenarios, as demonstrated later, this assumption is eliminated.

The objective of this paper is to design an algorithm that can solve distributed online convex optimization with adversarial constraints, ensuring that both network regret and cumulative constraint violation bounds grow sublinearly under the aforementioned assumptions. More importantly, this paper aims to show reduced (network) cumulative constraint violation bounds can be achieved under Slater's condition, thereby addressing the open problem in the literature. We formally introduce Slater's condition as follows.

\begin{assumption}\label{online_op:assgt}
(Slater's condition) There exists a point $x_s\in \mathbb{X}$ and a constant $\epsilon_s>0$ such that
\begin{align}\label{online_op:gtcon}
g_{t}(x_s)\le-\epsilon_s{\bf 1}_{m},~t\in\mathbb{N}_+.
\end{align}
\end{assumption}

Slater's condition is a sufficient condition for strong duality to hold in convex optimization problems \cite{boyd2004convex}, which has also been utilized in online convex optimization problems to show reduced constraint violation bounds can be achieved, e.g., \cite{yu2020lowJMLR,yu2017online,neely2017online,yi2020distributed}. However, to the best of our knowledge, existing literature has not achieved reduced (network) cumulative constraint violation bounds under Slater's condition. The key reason for this is that existing literature studying (network) cumulative constraint violation, e.g., \cite{NIPS2018_7852,yuan2021distributed,yuan2021distributedb,yi2021regretTAC}, often replaces the original constraint function with the clipped constraint function, which makes Slater's condition ineffective.
Therefore, the main challenge lies in addressing the contradiction between the need for the clipping operation when using the stricter cumulative constraint violation metric, as defined in \eqref{online_op:regc1} and \eqref{online_op:regc}, and the fact that it renders Slater's condition ineffective.

\section{Algorithm Description}\label{online_opsec:algorithm}
In this section, we propose a novel algorithm for the distributed online convex optimization problem with adversarial constraints as introduced in the previous section, and analyze its performance in the next section.

Recall that at the $t$-th round, the global loss and constraint functions are $l_t$ and $g_t$, respectively. The associated regularized Lagrangian function is
\begin{align*}
	\calL_{t}(x_t,q_t)&:=\frac{1}{n}\sum_{i=1}^{n} f_{i,t}(x_t)+r_t(x_t)
	+q_t^\top g_t(x_t)-\frac{1}{2\gamma_{t}}\|q_t\|^2,
\end{align*}
where $x_t\in\mathbb{R}^p$ and $q_t\in\mathbb{R}^m_+$ represent the primal and dual variables, respectively, and $\gamma_{t}$ is the regularization parameter. Instead of using the projected gradient ascent
\begin{align*}
	q_{t+1}=\Big[q_t+\alpha_{t}\frac{\partial \calL_{t}(x_t,q)}{\partial q}\Big|_{q=q_t}\Big]_+
\end{align*}
to update the dual variable, where $\alpha_{t}>0$ is the stepsize, we update it by directly maximizing $\calL_{t}(x_t,q)$ over all $q\in\mathbb{R}^m_+$, i.e.,
\begin{align}\label{online_op:al_c_q}
	q_{t+1}=\argmax_{q\in\mathbb{R}^m_+} \calL_{t}(x_t,q)=\gamma_{t}[g_t(x_t)]_+,
\end{align}
which follows the idea of updating dual variables in \cite{NIPS2018_7852}. The same idea has also been adopted in \cite{yuan2021distributedb}. Moreover, instead of using the  projected gradient descent
\begin{align*}
	x_{t+1}=\calP_\mathbb{X}\Big(x_t-\alpha_{t}\frac{\partial \calL_{t}(x,q_{t+1})}{\partial x}\Big|_{x=x_t}\Big)
\end{align*}
to update the primal variable,  we update it as follows:
\begin{align}\label{online_op:al_c_x}
	x_{t+1}=\argmin_{x\in\mathbb{X}}\{&\alpha_{t}\langle x,\omega_{t+1}\rangle+\alpha_{t}r_{t}(x)+\calD_\psi(x,x_{t})\},
\end{align}
where
\begin{align*}
	\omega_{t+1}=\frac{1}{n}\sum_{i=1}^{n}\nabla f_{i,t}(x_{t})+(\nabla g_{t}(x_{t}))^\top q_{t+1}.
\end{align*}
The updating rule \eqref{online_op:al_c_x} is inspired by  the composite objective mirror descent in \cite{duchi2010composite}, which has also been adopted in \cite{yuan2020distributed,yi2020distributed}. In the following, we introduce how to implement the updating rules \eqref{online_op:al_c_q}--\eqref{online_op:al_c_x} in a distributed manner.

We use $x_{i,t}$ to denote the local copy of the primal variable $x_t$. If we rewrite the dual variable in an agent-wise manner, i.e.,
$q_t=\col(q_{1,t},\dots,q_{n,t})$
with each $q_{i,t}\in\mathbb{R}^{m_i}_+$, then the updating rule \eqref{online_op:al_c_q} can be executed in an agent-wise manner as \eqref{online_op:al_q}. Note that $\omega_{i,t+1}$ defined in \eqref{online_op:al_bigomega} can be understood as a portion of $\omega_{t+1}$ that is available to agent $i$. Then, each $z_{i,t+1}$ updated by \eqref{online_op:al_x} can be understood as a local estimate of $x_{t+1}$ updated by \eqref{online_op:al_c_x}. In this case, for each agent $i$, $x_{i,t+1}$ computed by the consensus protocol \eqref{online_op:al_z} is used to track the average  $\frac{1}{n}\sum_{i=1}^nz_{i,t+1}$, and thus estimates $x_{t+1}$ more accurately. As a result, the updating rules \eqref{online_op:al_c_q}--\eqref{online_op:al_c_x}  can be executed in a distributed manner, which is summarized in pseudo-code as Algorithm~\ref{online_op:algorithm}. This algorithm is called the distributed online primal--dual composite mirror descent algorithm.

\begin{algorithm}[tb]
\caption{Distributed Online Primal--Dual Composite Mirror Descent}
\begin{algorithmic}\label{online_op:algorithm}
\STATE \textbf{Input}:   nonincreasing sequence $\{\alpha_t>0\}$ and nondecreasing sequence $\{\gamma_t>0\}$; differentiable and strongly convex function $\psi$.
\STATE \textbf{Initialize}:  $z_{i,1}\in\mathbb{X}$ for all $i\in[n]$.
\FOR{$t=1,\dots$}
\FOR{$i=1,\dots,n$  in parallel}
\STATE  Broadcast $z_{i,t}$ to $\mathcal{N}^{\outin}_i(\mathcal{G}_{t})$ and receive $z_{j,t}$ from $j\in\mathcal{N}^{\inout}_i(\mathcal{G}_{t})$.
\STATE  Select
\begin{align}
x_{i,t}=\sum_{j=1}^n[W_{t}]_{ij}z_{j,t}.\label{online_op:al_z}
\end{align}
\STATE  Observe $\nabla f_{i,t}(x_{i,t})$, $\nabla g_{i,t}(x_{i,t})$, $g_{i,t}(x_{i,t})$, and $r_{t}(\:\cdot\:)$.
\STATE  Update
\begin{subequations}
\begin{align}
q_{i,t+1}&=\gamma_{t}[g_{i,t}(x_{i,t})]_+,\label{online_op:al_q}\\
\omega_{i,t+1}&=\nabla f_{i,t}(x_{i,t})+(\nabla g_{i,t}(x_{i,t}))^\top q_{i,t+1},\label{online_op:al_bigomega}\\
z_{i,t+1}&=\argmin_{x\in\mathbb{X}}\{\alpha_{t}\langle x,\omega_{i,t+1}\rangle+\alpha_{t}r_{t}(x)+\calD_\psi(x,x_{i,t})\}.\label{online_op:al_x}
\end{align}
\end{subequations}
\ENDFOR
\ENDFOR
\STATE  \textbf{Output}: $\{x_{i,t}\}$.
\end{algorithmic}
\end{algorithm}


The minimization problem (\ref{online_op:al_x}) is strongly convex, so it can be solved with a linear convergence rate and closed-form solutions are available in special cases. For example, if $r_t$ is a linear mapping and the Euclidean distance is used as the Bregman distance,  i.e., $\calD_\psi(x,y)=\|x-y\|^2$, then the convex minimization problem (\ref{online_op:al_x}) can be solved by the projection
\begin{align*}
	z_{i,t+1}=\calP_\mathbb{X}\Big(x_{i,t}-\frac{\alpha_{t}}{2}(\omega_{i,t+1}+\nabla r_{t})\Big).
\end{align*}

To end this section, we would like to emphasize the key novelty of the proposed algorithm. Different from previous approaches in \cite{NIPS2018_7852,yuan2021distributed,yuan2021distributedb,yi2021regretTAC}, our algorithm does not simply replace the original constraint function with the clipped constraint function. Instead, it utilizes the clipped constraint function solely for updating the dual variables, derived from directly maximizing the regularized Lagrangian function, see \eqref{online_op:al_c_q} and \eqref{online_op:al_q}. Although this may seem like a minor difference, it guarantees the achievement of state-of-the-art (network) cumulative constraint violation bounds without Slater's condition. More importantly, it serves as the key to solving the open problem of achieving reduced (network) cumulative constraint violation bounds under Slater's condition through properly designed analysis, since it enables the use of the stricter cumulative constraint violation metric while preserving the effectiveness of Slater's condition. The detailed explanations are elaborated in the next section.

\section{Performance Analysis}\label{online_opsec:main}
This section analyzes network regret and cumulative constraint violation bounds for Algorithm~\ref{online_op:algorithm} under different scenarios.

\subsection{Preliminary Results}
We first bound local regret and (squared) cumulative constraint violation, the accumulated (squared) consensus error, and the changes caused by composite mirror descent in the following.
\begin{lemma}\label{online_op:theoremreg-local}
Suppose Assumptions~\ref{online_op:assfunction}--\ref{online_op:assbregman} hold. For all $i\in[n]$, let $\{x_{i,t}\}$ be the sequences generated by Algorithm~\ref{online_op:algorithm} with $\gamma_t=\gamma_0/\alpha_t$, where $\gamma_0\in(0,\sigma/(4G_2^2)]$ is a constant. Then, for any $T\in\mathbb{N}_+$,
\begin{subequations}
\begin{align}
&\frac{1}{n}\sum_{i=1}^{n}\sum_{t=1}^{T}\Big(l_{i,t}(x_{i,t})-l_{i,t}(y)+\frac{\sigma\|\epsilon^z_{i,t}\|^2}{4\alpha_{t}}\Big)\le \sum_{t=1}^{T}\frac{2G_1^2\alpha_{t}}{\sigma}+\frac{1}{ n}\sum_{t=1}^{T}\sum_{i=1}^n\Delta_{i,t}(y),~\forall y\in\mathcal{X}_T,
\label{online_op:theoremregequ-local}\\
&\sum_{t=1}^T\sum_{i=1}^n\frac{1}{2}\Big(\frac{q_{i,t+1}^\top g_{i,t}(x_{i,t})}{\gamma_t}+\frac{\sigma\|\epsilon^z_{i,t}\|^2}{2\gamma_0}\Big)\le h_T(y)+\tilde{h}_T(y),~\forall y\in\mathcal{X}_T,\label{online_op:theoremregequ2_g2}\\
&\frac{1}{n}\sum_{t=1}^{T}\sum_{i=1}^n\sum_{j=1}^n\|x_{i,t}-x_{j,t}\|
\le n\varepsilon_1+\tilde{\varepsilon}_2\sum_{t=1}^{T}\sum_{i=1}^n
\|\epsilon^z_{i,t}\|,\label{online_op:xitxbarT2_slater}\\
&\frac{1}{n}\sum_{t=1}^{T}\sum_{i=1}^n\sum_{j=1}^n\|x_{i,t}-x_{j,t}\|^2\
\le \tilde{\varepsilon}_3+\tilde{\varepsilon}_4\sum_{t=1}^T\sum_{i=1}^n\|\epsilon^z_{i,t}\|^2,\label{online_op:xitxbargsquar}\\
&\|\epsilon^z_{i,t}\|\le\frac{1}{\sigma}(G_2\gamma_{0}\|[g_{i,t}(x_{i,t})]_+\|+G_1\alpha_{t}),\label{online_op:xxt}
\end{align}
\end{subequations}
where
\begin{align*}
&\Delta_{i,t}(y)=\frac{1}{\alpha_{t}}(\calD_\psi(y,x_{i,t})
-\calD_\psi(y,x_{i,t+1})),~\epsilon^z_{i,t}=z_{i,t+1}-x_{i,t},\\
&h_T(y)=\sum_{i=1}^n\sum_{t=1}^T\frac{q_{i,t+1}^\top g_{i,t}(y)}{\gamma_t},
~\tilde{h}_T(y)=\sum_{t=1}^T\frac{nF}{\gamma_{t}}+\sum_{t=1}^T\frac{2n\gamma_0G_1^2}{\sigma\gamma_{t}^2}+\sum_{i=1}^n\frac{\calD_\psi(y,x_{i,1})}{\gamma_0},\\
&\varepsilon_1=\frac{2\tau}{\lambda(1-\lambda)}\sum_{i=1}^n\|z_{i,1}\|,~\tilde{\varepsilon}_2=\frac{2(n\tau +2-2\lambda)}{1-\lambda},~\tilde{\varepsilon}_3=\frac{16n\tau^2}{\lambda^{2}(1-\lambda^{2})}\Big(\sum_{i=1}^n\|z_{i,1}\|\Big)^2,\\
&\tilde{\varepsilon}_4=\frac{16n^2\tau^2}{(1-\lambda)^2}+32,~\tau=(1-w/4n^2)^{-2}>1,~\lambda=(1-w/4n^2)^{1/ B }\in(0,1).
\end{align*}
\end{lemma}
\begin{proof}
	See Appendix~\ref{online_op:theoremregproof-local}.
\end{proof}

As mentioned in the previous section, the key innovation of Algorithm~\ref{online_op:algorithm} lies in the utilization of the clipped constraint function solely for updating the dual variables, rather than directly replacing the original constraint function.
With Lemma~\ref{online_op:theoremreg-local}, in the subsequent discussion, we will demonstrate how to obtain network regret and cumulative constraint violation bounds for Algorithm~\ref{online_op:algorithm}. We will specifically highlight the significance of the aforementioned novelty in achieving (network) cumulative constraint violation bounds in both scenarios: without and with Slater's condition. Additionally, we will provide an elucidation of why (network) cumulative constraint violation bounds can be reduced under Slater's condition.

Firstly, noting that \eqref{online_op:theoremregequ-local} and \eqref{online_op:xitxbarT2_slater} respectively provide an upper bound for local regret and the accumulated consensus error, an upper bound for network regret can be derived by combining \eqref{online_op:theoremregequ-local}, \eqref{online_op:xitxbarT2_slater}, and \eqref{online_op:assfunction:functionLipf}.

Secondly, noting that
\begin{align}
\frac{q_{i,t+1}^\top g_{i,t}(x_{i,t})}{\gamma_t}=[g_{i,t}(x_{i,t})]_+^\top g_{i,t}(x_{i,t})=\|g_{i,t}(x_{i,t})\|^2\label{online_op:qibound_ab}
\end{align}
due to \eqref{online_op:al_q} and the fact that $a^\top[a]_+=\|[a]_+\|^2$ for any vector $a$, we know that \eqref{online_op:theoremregequ2_g2} provides an upper bound for local squared cumulative constraint violation. Moreover, note that \eqref{online_op:xitxbargsquar} provides an upper bound for the accumulated squared consensus error, which is $\mathcal{O}(\tilde{h}_T(y))$ since $h_T(y)\le0,~\forall y\in\mathcal{X}_T$. Then, we can get an upper bound for network squared cumulative constraint violation by combining \eqref{online_op:theoremregequ2_g2}, \eqref{online_op:xitxbargsquar}, and \eqref{online_op:assfunction:functionLipg}. As a result, an $\mathcal{O}((T\tilde{h}_T(y))^{1/2})$ bound for network cumulative constraint violation can be derived since
\begin{align}\label{online_op:constrain_vio_def}
\Big(\frac{1}{n}\sum_{i=1}^n\sum_{t=1}^T\|[g_{t}(x_{i,t})]_+\|\Big)^2
\le\frac{T}{n}\sum_{i=1}^n\sum_{t=1}^T\|[g_{t}(x_{i,t})]_+\|^2,
\end{align}
which holds due to the H\"{o}lder's inequality.

Thirdly, when Slater's condition holds, noting that
\begin{align}
h_T(x_s)&=\sum_{i=1}^n\sum_{t=1}^T\frac{q_{i,t+1}^\top g_{i,t}(x_s)}{\gamma_t}
=\sum_{i=1}^n\sum_{t=1}^T[g_{i,t}(x_{i,t})]_+^\top g_{i,t}(x_s)\nonumber\\
&\le-\sum_{i=1}^n\sum_{t=1}^T\epsilon_s[g_{i,t}(x_{i,t})]_+^\top {\bf 1}_{m_i}
=-\epsilon_s\sum_{i=1}^n\sum_{t=1}^T\|[g_{i,t}(x_{i,t})]_+\|_1\nonumber\\
&\le-\epsilon_s\sum_{i=1}^n\sum_{t=1}^T\|[g_{i,t}(x_{i,t})]_+\|,\label{online_op:gtcon_hT}
\end{align}
where the second equality holds due to \eqref{online_op:al_q} and the first inequality holds due to \eqref{online_op:gtcon}, we know that \eqref{online_op:theoremregequ2_g2} provides an upper bound for local cumulative constraint violation, which is $\mathcal{O}(\tilde{h}_T(x_s))$. On the other hand, from \eqref{online_op:xitxbarT2_slater}, \eqref{online_op:assfunction:functionLipg}, and \eqref{online_op:xxt}, we know that network cumulative constraint violation can be bounded by local cumulative constraint violation. As a result, we can get an $\mathcal{O}(\tilde{h}_T(x_s))$ bound for network cumulative constraint violation.

Note that $\mathcal{O}(\tilde{h}_T(x_s))$, the upper bound for network cumulative constraint violation under Slater's condition, has the same order with respect to $T$ as $\mathcal{O}(\tilde{h}_T(y))$, and thus it has smaller order with respect to $T$ than $\mathcal{O}((T\tilde{h}_T(y))^{1/2})$, the upper bound for (network) cumulative constraint violation without Slater's condition, since it can be guaranteed that $\mathcal{O}(\tilde{h}_T(y))=\mathbf{o}(T)$ by appropriately selecting the stepsize sequence $\{\alpha_t\}$. Therefore, reduced (network) cumulative constraint violation bounds can be achieved under Slater's condition.

With the explanations above,  network regret and cumulative constraint violation bounds for the general cases are provided in the following lemma.
\begin{lemma}\label{online_op:theoremreg}
Under the same condition as stated in Lemma~\ref{online_op:theoremreg-local}, for any $T\in\mathbb{N}_+$, it holds that
\begin{subequations}
\begin{align}
&\frac{1}{n}\sum_{i=1}^{n}\sum_{t=1}^{T}l_{t}(x_{i,t})-\sum_{t=1}^{T}l_{t}(y)
\le \varepsilon_1G_1+\sum_{t=1}^{T}\varepsilon_2\alpha_{t}+\frac{1}{ n}\sum_{t=1}^{T}\sum_{i=1}^n\Delta_{i,t}(y),~\forall y\in\mathcal{X}_T,
\label{online_op:theoremregequ}\\
&\frac{1}{n}\sum_{i=1}^n\sum_{t=1}^T\|[g_{t}(x_{i,t})]_+\|\le\sqrt{\varepsilon_3T+\varepsilon_4T\tilde{h}_T(y)},~\forall y\in\mathcal{X}_T,\label{online_op:theoremconsequ}\\
&\frac{1}{n}\sum_{i=1}^n\sum_{t=1}^T\|[g_{t}(x_{i,t})]_+\|
\le n\varepsilon_1G_2+\varepsilon_{5}\sum_{t=1}^T\alpha_t+ \varepsilon_{6}\sum_{t=1}^T\sum_{i=1}^n\|[g_{i,t}(x_{i,t})]_+\|,\label{online_op:lemma-g-slater}
\end{align}
\end{subequations}
where
\begin{align*}
&\varepsilon_2=\frac{(\tilde{\varepsilon}_2^2+2)G_1^2}{\sigma},
~\varepsilon_3=2G_2^2\tilde{\varepsilon}_3,~\varepsilon_4=\frac{4\max\{1,G_2^2\tilde{\varepsilon}_4\}}{\min\{1,\frac{\sigma}{2\gamma_0}\}},~\varepsilon_{5}=\frac{n\tilde{\varepsilon}_2G_1G_2}{\sigma },~\varepsilon_{6}=\frac{\tilde{\varepsilon}_2G_2^2\gamma_{0}+\sigma}{\sigma}.
\end{align*}
\end{lemma}
\begin{proof}
	See Appendix~\ref{online_op:theoremregproof}.
\end{proof}

\subsection{Convex Scenario}
In this section, we show that network regret and cumulative constraint violation bounds grow sublinearly if the natural vanishing stepsize sequence is used.

\begin{theorem}\label{online_op:corollaryreg}
Suppose Assumptions~\ref{online_op:assfunction}--\ref{online_op:assbregmanlip} hold. For all $i\in[n]$, let $\{x_{i,t}\}$ be the sequences generated by Algorithm~\ref{online_op:algorithm} with
\begin{align}\label{online_op:stepsize1}
\alpha_t=\frac{1}{t^{c}},~\gamma_t=\frac{\gamma_0}{\alpha_t},~\forall t\in\mathbb{N}_+,
\end{align} where $c\in(0,1)$ and $\gamma_0\in(0,\sigma/(4G_2^2)]$ are constants. Then, for any $T\in\mathbb{N}_+$,
\begin{subequations}\label{online_op:corollaryreg-bound}
\begin{align}
&\NetReg(T)=\mathcal{O}(T^{\max\{c,1-c\}}),\label{online_op:corollaryregequ1}\\
&\frac{1}{n}\sum_{i=1}^n\sum_{t=1}^T\|[g_{t}(x_{i,t})]_+\|=\mathcal{O}(T^{1-c/2}).
\label{online_op:corollaryconsequ}
\end{align}
\end{subequations}
\end{theorem}
\begin{proof}
The explicit expressions of the right-hand sides of \eqref{online_op:corollaryreg-bound} and the proof are given in  Appendix~\ref{online_op:corollaryregproof}.
\end{proof}

Theorem~\ref{online_op:corollaryreg} shows that Algorithm~\ref{online_op:algorithm} generalizes the results in \cite{sun2017safety,NIPS2018_7852,yuan2021distributed,yuan2021distributedb}.
Specifically, by setting $c=0.5$ in Theorem~\ref{online_op:corollaryreg}, the result in \cite{sun2017safety} is recovered, although the algorithm proposed in \cite{sun2017safety} is centralized and the standard constraint violation metric rather than the stricter metrics is used. The bounds achieved in \eqref{online_op:corollaryreg-bound} are consistent with the results in \cite{NIPS2018_7852,yuan2021distributed,yuan2021distributedb}, although in \cite{NIPS2018_7852,yuan2021distributed,yuan2021distributedb} the constraint functions are time-invariant and known in advance. Moreover, in \cite{NIPS2018_7852} the proposed algorithm is centralized, and in \cite{yuan2021distributed} the loss functions are quadratic and the constraint functions are linear.
Theorem~\ref{online_op:corollaryreg} also shows that Algorithm~\ref{online_op:algorithm} achieves improved performance compared with the distributed online algorithm proposed in \cite{yuan2017adaptive}, although the global constraint functions in \cite{yuan2017adaptive} are time-invariant and known in advance by each agent. Specifically, $\NetReg(T)=\mathcal{O}(T^{0.5+\beta})$ and $\frac{1}{n}\sum_{i=1}^n\|[\sum_{t=1}^Tg(x_{i,t})]_+\|=\mathcal{O}(T^{1-\beta/2})$ were achieved in \cite{yuan2017adaptive}, where $\beta\in(0,0.5)$. The same bounds as shown in \eqref{online_op:corollaryreg-bound} have also been achieved by the distributed online algorithm proposed in \cite{yi2021regretTAC}. Compared to the algorithm in \cite{yi2021regretTAC}, the potential drawback of Algorithm~\ref{online_op:algorithm} is that it uses $G_2$, the uniform bound for the gradients of the local constraint functions\footnote{This bound is available if the global constraint function is time-invariant and known in advance by each agent as assumed in \cite{yuan2017adaptive,yuan2021distributed,yuan2021distributedb}.},  to design the algorithm parameter $\gamma_0$. However, in the following, we also show that Algorithm~\ref{online_op:algorithm}  can achieve reduced network cumulative constraint violation bounds when Slater's condition holds, while  \cite{sun2017safety,NIPS2018_7852,yuan2017adaptive,yuan2021distributed,yuan2021distributedb,yi2021regretTAC} do not have such a result.

\subsection{Slater's Condition Scenario}
In this section, we show that Algorithm~\ref{online_op:algorithm} achieves a reduced network cumulative constraint violation bound under Slater's condition.

\begin{theorem}\label{online_op:corollaryreg_slater}
Suppose Assumptions~\ref{online_op:assfunction}--\ref{online_op:assgt} hold. For all $i\in[n]$, let $\{x_{i,t}\}$ be the sequences generated by Algorithm~\ref{online_op:algorithm} with \eqref{online_op:stepsize1}. Then, for any $T\in\mathbb{N}_+$,
\begin{subequations}\label{online_op:corollaryreg_slater-bound}
\begin{align}
&\NetReg(T)=\mathcal{O}(T^{\max\{c,1-c\}}),
\label{online_op:corollaryregequ1_slater}\\
&\frac{1}{n}\sum_{i=1}^n\sum_{t=1}^T\|[g_{t}(x_{i,t})]_+\|=\mathcal{O}(T^{1-c}).
\label{online_op:corollaryconsequ_slater}
\end{align}
\end{subequations}
\end{theorem}
\begin{proof}
The explicit expressions of the right-hand sides of \eqref{online_op:corollaryreg_slater-bound} and the proof are given in  Appendix~\ref{online_op:corollaryregproof_slater}.
\end{proof}

Under Slater's condition, the centralized algorithms proposed in  \cite{yu2017online,neely2017online} achieved $\mathcal{O}(\sqrt{T})$ regret  and  constraint violation bounds. In \cite{NIPS2018_7852}, it was emphasized that extending the results in \cite{yu2017online,neely2017online} to incorporate the stricter cumulative constraint violation metric remains unclear. This is due to the fact that the clipped constraint functions $\{[g_t]_+\}$ fail to satisfy Slater's condition, even if the original constraint functions $\{g_t\}$ do satisfy it. It is an open problem how to achieve reduced cumulative constraint violation bounds for online convex optimization with adversarial constraints under Slater's condition. Theorem~\ref{online_op:corollaryreg_slater} shows that Algorithm~\ref{online_op:algorithm} extends the results presented in \cite{yu2017online,neely2017online}, thereby solving the aforementioned open problem. Specifically, the results in \cite{yu2017online,neely2017online} are recovered when setting $c=0.5$ in Theorem~\ref{online_op:corollaryreg_slater}.


\subsection{Strongly Convex Scenario}
The network regret bound provided in Theorems~\ref{online_op:corollaryreg} and \ref{online_op:corollaryreg_slater} is at least $\mathcal{O}(\sqrt{T})$ and it can be reduced to strictly less than $\mathcal{O}(\sqrt{T})$ if the local loss functions are strongly convex. Without loss of generality, we assume the private parts $\{f_{i,t}(x)\}$ are strongly convex.
\begin{assumption}\label{online_op:assstrongconvex}
For any $i\in[n]$ and $t\in\mathbb{N}_+$, $\{f_{i,t}(x)\}$ are strongly convex over $\mathbb{X}$ with respect to $\psi$ with $\mu>0$, i.e., for all $x,y\in\mathbb{X}$,
\begin{align}\label{online_op:assstrongconvexequ}
f_{i,t}(x)\ge f_{i,t}(y)+\langle x-y,\nabla f_{i,t}(y)\rangle+\mu\calD_{\psi}(x,y).
\end{align}
\end{assumption}

\begin{theorem}\label{online_op:corollaryreg_sc}
Suppose Assumptions~\ref{online_op:assfunction}--\ref{online_op:assbregmanlip} and \ref{online_op:assstrongconvex} hold. For all $i\in[n]$, let $\{x_{i,t}\}$ be the sequences generated by Algorithm~\ref{online_op:algorithm} with \eqref{online_op:stepsize1}. Then, for any $T\in\mathbb{N}_+$,
\begin{subequations}\label{online_op:corollaryreg_sc-bound}
\begin{align}
&\NetReg(T)=\mathcal{O}(T^{1-c}),\label{online_op:corollaryregequ1_sc}\\
&\frac{1}{n}\sum_{i=1}^n\sum_{t=1}^T\|[g_{t}(x_{i,t})]_+\|=\mathcal{O}(T^{1-c/2}).\label{online_op:corollaryconsequ_sc}
\end{align}
\end{subequations}
Moreover, if Assumption~\ref{online_op:assgt} also holds, then
\begin{align}
\frac{1}{n}\sum_{i=1}^n\sum_{t=1}^T\|[g_{t}(x_{i,t})]_+\|=\mathcal{O}(T^{1-c}).
\label{online_op:corollaryconsequ_slater_sc}
\end{align}
\end{theorem}
\begin{proof}
The explicit expressions of the right-hand sides of \eqref{online_op:corollaryreg_sc-bound}--\eqref{online_op:corollaryconsequ_slater_sc} and the proof are given in  Appendix~\ref{online_op:corollaryregproof_sc}.
\end{proof}

The same bounds as shown in \eqref{online_op:corollaryreg_sc-bound} have also been achieved by the distributed online algorithm proposed in \cite{yi2021regretTAC}. It should be pointed out that  Algorithm~\ref{online_op:algorithm} achieves a reduced network cumulative constraint violation bound under Slater's condition as shown in \eqref{online_op:corollaryconsequ_slater_sc}, while the algorithm in \cite{yi2021regretTAC} does not have such a property.
Both network regret and cumulative constraint violation bounds can be further reduced if the convex parameter $\mu$ is known in advance.
\begin{theorem}\label{online_op:corollaryreg_scmu}
Suppose Assumptions~\ref{online_op:assfunction}--\ref{online_op:assbregman} and \ref{online_op:assstrongconvex} hold. For all $i\in[n]$, let $\{x_{i,t}\}$ be the sequences generated by Algorithm~\ref{online_op:algorithm} with
\begin{align}\label{online_op:stepsize1scmu}
\alpha_t=\frac{1}{\mu t},~\gamma_t=\frac{\gamma_0}{\alpha_t},~\forall t\in\mathbb{N}_+,
\end{align} where $\gamma_0\in(0,\sigma/(4G_2^2)]$ is a constant. Then, for any $T\in\mathbb{N}_+$,
\begin{subequations}\label{online_op:corollaryreg_scmu-bound}
\begin{align}
&\NetReg(T)=\mathcal{O}(\log(T)),\label{online_op:corollaryregequ1_scmu}\\
&\frac{1}{n}\sum_{i=1}^n\sum_{t=1}^T\|[g_{t}(x_{i,t})]_+\|=\mathcal{O}(\sqrt{\log(T)T}).
\label{online_op:corollaryconsequ_scmu}
\end{align}
\end{subequations}
Moreover, if Assumption~\ref{online_op:assgt} also holds, then
\begin{align}
\frac{1}{n}\sum_{i=1}^n\sum_{t=1}^T\|[g_{t}(x_{i,t})]_+\|=\mathcal{O}(\log(T)).\label{online_op:corollaryconsequ_slater_scmu}
\end{align}
\end{theorem}
\begin{proof}
The explicit expressions of the right-hand sides of \eqref{online_op:corollaryreg_scmu-bound}--\eqref{online_op:corollaryconsequ_slater_scmu} and the proof are given in  Appendix~\ref{online_op:corollaryregproof_scmu}.
\end{proof}

The bounds achieved in \eqref{online_op:corollaryreg_scmu-bound} are consistent with the results in \cite{NIPS2018_7852,yuan2021distributedb} for strongly convex loss functions, although the proposed algorithm in \cite{NIPS2018_7852} is centralized, and the constraint functions in \cite{NIPS2018_7852,yuan2021distributedb} are time-invariant and known in advance. Moreover, compared \eqref{online_op:corollaryreg_scmu-bound} with the results that $\NetReg(T)
=\mathcal{O}(T^c)$ and $\frac{1}{n}\sum_{i=1}^n\|[\sum_{t=1}^Tg(x_{i,t})]_+\|=\mathcal{O}(T^{1-c/2})$ as achieved in \cite{yuan2017adaptive} for strongly convex loss functions, we know that our Algorithm~\ref{online_op:algorithm} achieves improved performance, although the global constraint function in \cite{yuan2017adaptive} is time-invariant and known in advance by each agent. It should be highlighted that our Algorithm~\ref{online_op:algorithm} achieves an $\mathcal{O}(\log(T))$ network cumulative constraint violation bound under Slater's condition as shown in \eqref{online_op:corollaryconsequ_slater_scmu}, which, to the best of our knowledge, is achieved for the first time in the literature.

Note that Assumption~\ref{online_op:assbregmanlip} is not needed in Theorem~\ref{online_op:corollaryreg_scmu}.
This assumption is not needed in Theorems~\ref{online_op:corollaryreg}--\ref{online_op:corollaryreg_slater} either when knowing the total number of rounds $T$ in advance and choosing $\alpha_t=1/T^c$ in \eqref{online_op:stepsize1}. In this case, with slight modifications of the proof, we can show that the results stated in Theorems~\ref{online_op:corollaryreg}--\ref{online_op:corollaryreg_slater} still hold.

\section{SIMULATIONS}\label{online_opsec:simulation}
In this section, we verify the theoretical results through numerical simulations.

Note that, as pointed out in the previous section, without Slater's condition, the proposed algorithm achieves the same network regret and cumulative constraint violation bounds as those in \cite{yi2021regretTAC}. It is not this paper's goal to show that the proposed algorithm has better performance without Slater's condition. The key contribution of this paper is demonstrating that reduced (network) cumulative constraint violation bounds can be achieved when Slater's condition holds, which is a significant result not found in existing literature. Therefore, we only consider examples where the constraint function satisfies Slater's condition. Specifically, similar to \cite{yi2021regretTAC}, we consider distributed online linear regression problem with time-varying linear inequality constraints formulated as follows\footnote{For a fair comparison, we do not consider regularization in the loss, i.e., set $r_t(\:\cdot\:)\equiv0$.}:
\begin{mini*}
{x}{\sum_{t=1}^{T}\sum_{i=1}^{n}\frac{1}{2}\|H_{i,t}x-z_{i,t}\|^2}{}{}
\addConstraint{x\in \mathbb{X},~A_{i,t}x-a_{i,t}\le}{{\bm 0}_{m_i},}{~\forall i\in[n]~\forall t\in[T],}
\end{mini*}
where $H_{i,t}\in\mathbb{R}^{d_i\times p}$, $z_{i,t}\in\mathbb{R}^{d_i}$, $A_{i,t}\in\mathbb{R}^{m_i\times p}$, and $a_{i,t}\in\mathbb{R}^{m_i}$ with $d_i\in\mathbb{N}_+$. Each component of $H_{i,t}$ is generated from $\calU(-1,1)$ and $z_{i,t}=H_{i,t}\bm{1}_p+\varepsilon_{i,t}$, where $\varepsilon_{i,t}$ is a standard normal random vector. Each component of $A_{i,t}$ and $a_{i,t}$ is generated from $\calU(0,2)$ and $\calU(a,a+1)$, respectively, where $a>0$ is used to guarantee Slater's condition holds. At each time $t$, an undirected random graph is used as the communication graph. Specifically, connections between agents are random and the probability of two agents being connected is $\rho$. To guarantee that Assumption~\ref{online_op:assgraph} holds, edges $(i,i+1),~i\in[n-1]$ are also added and $[W_t]_{ij}=1/n$ if $(j,i)\in\mathcal{E}_t$ and $[W_t]_{ii}=1-\sum_{j=1}^n[W_t]_{ij}$.

\bgroup
\def\arraystretch{1.15}
\begin{table}[]
\caption{Input of each algorithm.}
\label{online_op::table-input}
\centering
\small
\begin{tabular}{M{3.6cm}|M{7.2cm}N}
\hline
Algorithms  & Inputs\\

\hline
Algorithm~\ref{online_op:algorithm}      & $\alpha_t=1/t$, $\gamma_t=0.15/\alpha_t$, $\psi(x)=\|x\|^2$ &\\

\hline
Algorithm~1 in  \cite{yu2017online}   & $\alpha=T$, $V=\sqrt{\alpha}$ &\\

\hline
Algorithm~1 in \cite{yi2021regretTAC}     & $\alpha_t=0.7/t$, $\beta_t=1/\sqrt{t}$, $\gamma_t=1/\sqrt{t}$&\\

\hline
\end{tabular}
\end{table}
\egroup

We compare Algorithm~\ref{online_op:algorithm} with the centralized algorithm in \cite{yu2017online} (which considers Slater's condition but uses the standard constraint violation metric) and the distributed algorithm in \cite{yi2021regretTAC} (which uses the stricter cumulative constraint violation but does not consider Slater's condition).  We set $n=100$, $d_i=4$, $p=10$, $m_i=2$, $\mathbb{X}=[-5,5]^p$, $a=0.01$, and  $\rho=0.1$. The inputs of these algorithms are listed in TABLE~\ref{online_op::table-input}. Fig.~\ref{online_op:fig:reg_alg} and Fig.~\ref{online_op:fig:cons_alg} illustrate the trajectories of the accumulated loss and the cumulative constraint violation, respectively. Fig.~\ref{online_op:fig:reg_alg} shows that these algorithms have almost the same accumulated loss, which is in accordance with the theoretical results. Fig.~\ref{online_op:fig:cons_alg} shows that the algorithm in \cite{yi2021regretTAC} has significantly smaller cumulative constraint violation than that in \cite{yu2017online}, which is reasonable since the standard constraint violation metric rather than the stricter cumulative constraint violation metric is used in \cite{yu2017online}. More importantly, Fig.~\ref{online_op:fig:cons_alg} also shows that our proposed algorithm has significantly smaller cumulative constraint violation than that in \cite{yi2021regretTAC}, which matches the theoretical results. As previously explained, the main reason for this is that in \cite{yi2021regretTAC}, the clipped constraint function is directly used to replace the original constraint function, which renders Slater's condition ineffective. In contrast, in this paper, the clipped constraint function is solely used for updating the dual variables, which enables the use of the stricter cumulative constraint violation metric while preserving the effectiveness of Slater's condition.

\begin{figure}[]
  \centering
  \includegraphics[width=0.7\linewidth]{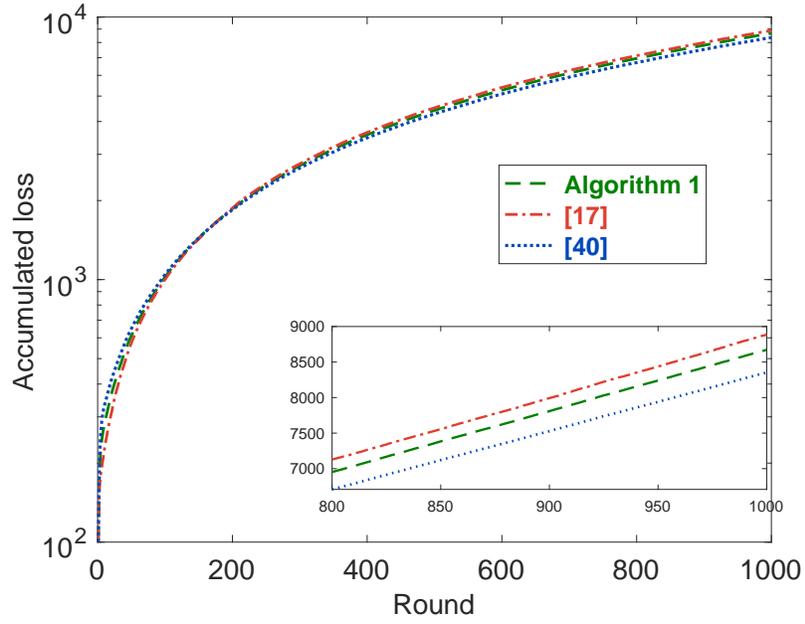}
\caption{Trajectories of the accumulated loss.}
\label{online_op:fig:reg_alg}
\end{figure}

\begin{figure}[!]
  \centering
  \includegraphics[width=0.7\linewidth]{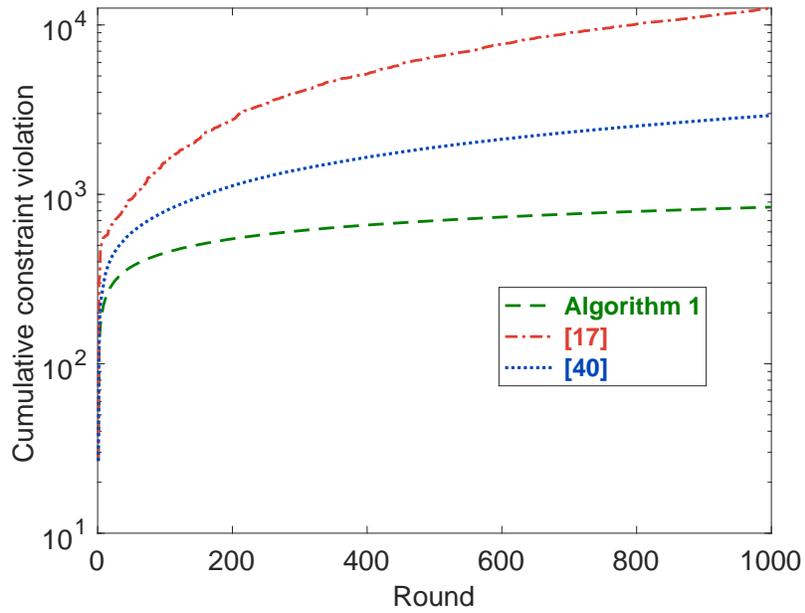}
\caption{Trajectories of the cumulative constraint violation.}
\label{online_op:fig:cons_alg}
\end{figure}

\section{CONCLUSION}\label{online_opsec:conclusion}

In this paper, we addressed the problem of distributed online convex optimization with adversarial constraints. We proposed a novel distributed online algorithm to solve this problem and conducted an analysis of network regret and cumulative constraint violation bounds for the proposed algorithm under different scenarios. Notably, we demonstrated that, for the first time, reduced (network) cumulative constraint violation bounds can be achieved under Slater's condition, both in convex and strongly convex scenarios. In the future, we plan to explore the challenging bandit setting since in various applications only the values of the loss and constraint functions are available. Additionally, we aim to tackle the issue of reducing communication complexity, as previous research has highlighted the significant impact of communication on distributed algorithms.


\appendix\label{online_op:appendix}

\subsection{Useful Lemmas}
We first present some results on the regularized Bregman projection.
\begin{lemma}\label{online_op:lemma_mirror} (Lemma~1 in \cite{yi2020distributed})
	Suppose that $h:\mathbb{X}\rightarrow\mathbb{R}$ is a convex function and $\nabla h(x),~\forall x\in\mathbb{X}$, exists. Then, for any $y,z\in\mathbb{X}$,
	the regularized Bregman projection
	\begin{align*}
		\tilde{x}=\argmin_{x\in\mathbb{X}}\{h(x)+\calD_\psi(x,z)\}
	\end{align*}
	satisfies the following inequalities
	\begin{subequations}
		\begin{align}
			\langle \nabla h(\tilde{x}), \tilde{x}-y\rangle&\le \calD_\psi(y,z)-\calD_\psi(y,\tilde{x})-\calD_\psi(\tilde{x},z),\label{online_op:lemma_mirroreq2}\\
			\|\tilde{x}-z\|&\le\frac{\|\nabla h(z)\|}{\sigma}.\label{online_op:lemma_mirrorine}
		\end{align}
	\end{subequations}
\end{lemma}

We next quantify the disagreement among the local temporary primal variables $\{z_{i,t}\}$.
\begin{lemma}\label{online_op:lemma_neterror}
If Assumption~\ref{online_op:assgraph} holds. For all $i\in[n]$ and $t\in\mathbb{N}_+$, $z_{i,t}$ generated by Algorithm~\ref{online_op:algorithm} satisfy
\begin{align}
\|z_{i,t}-\bar{z}_{t}\|
&\le \tau \lambda^{t-2}\sum_{j=1}^n\|z_{j,1}\|
+\tau\sum_{s=1}^{t-2}\lambda^{t-s-2}\sum_{j=1}^n\|\epsilon^z_{j,s}\|
+\|\epsilon^z_{i,t-1}\|+\frac{1}{n}\sum_{j=1}^n\|\epsilon^z_{j,t-1}\|,\label{online_op:xitxbar}
\end{align}
where $\bar{z}_{t}=\frac{1}{n}\sum_{i=1}^nz_{i,t}$.
\end{lemma}
\begin{proof}
From \eqref{online_op:al_z} and $\epsilon^z_{i,t-1}=z_{i,t}-x_{i,t-1}$, we have
\begin{align*}
z_{i,t}=\sum_{j=1}^n[W_{t-1}]_{ij}z_{j,t-1}+\epsilon^z_{i,t-1}.
\end{align*}
Then, following the proof of Lemma~4 in \cite{yi2021regretTAC}, we know that the result holds.
\end{proof}

We finally analyze regret at one round.
\begin{lemma}\label{online_op:lemma_regretdelta}
Suppose Assumptions~\ref{online_op:assfunction} and \ref{online_op:ass_subgradient}--\ref{online_op:assbregman} hold. For all $i\in[n]$, let $\{x_{i,t}\}$ be the sequences generated by Algorithm \ref{online_op:algorithm} and $y$ be an arbitrary point in $\mathbb{X}$, then
\begin{align}\label{online_op:lemma_regretdeltaequ}
&\frac{1}{ n}\sum_{i=1}^nq_{i,t+1}^\top g_{i,t}(x_{i,t})+\frac{1}{n}\sum_{i=1}^n(l_{i,t}(x_{i,t})-l_{i,t}(y))\le \frac{1}{n}\sum_{i=1}^nq_{i,t+1}^\top g_{i,t}(y)+\frac{\tilde{\Delta}_t}{n}+\frac{1}{ n}\sum_{i=1}^n\Delta_{i,t}(y),
\end{align}
where
\begin{align*}
\tilde{\Delta}_t=\sum_{i=1}^n(G_1+G_2\|q_{i,t+1}\|)\|\epsilon^z_{i,t}\|
-\sum_{i=1}^n\frac{\sigma\|\epsilon^z_{i,t}\|^2}{2\alpha_{t}}.
\end{align*}
\end{lemma}
\begin{proof}
From Assumptions~\ref{online_op:assfunction} and \ref{online_op:ass_subgradient}, we have
\begin{subequations}
\begin{align}
f_{i,t}(y)&\ge f_{i,t}(x)+\langle \nabla f_{i,t}(x),y-x\rangle,~\forall x,y\in\mathbb{X},\label{online_op:subgradient1}\\
r_t(y)&\ge r_t(x)+\langle \nabla r_t(x),y-x\rangle,~\forall x,y\in\mathbb{X},\label{online_op:subgradient2}\\
g_{i,t}(y)&\ge g_{i,t}(x)+\nabla g_{i,t}(x)(y-x),~\forall x,y\in\mathbb{X}.\label{online_op:subgradient3}
\end{align}
\end{subequations}

From (\ref{online_op:subgradient1}), (\ref{online_op:subgradient2}), and \eqref{online_op:subgupper}, we have
\begin{align}\label{online_op:fxy}
l_{i,t}(x_{i,t})-l_{i,t}(y)
&=f_{i,t}(x_{i,t})+r_t(x_{i,t})-f_{i,t}(y)-r_t(y)\nonumber\\
&=f_{i,t}(x_{i,t})-f_{i,t}(y)+r_t(x_{i,t})-r_t(z_{i,t+1})+r_t(z_{i,t+1})-r_t(y)\nonumber\\
&\le\langle\nabla f_{i,t}(x_{i,t}),x_{i,t}-y\rangle+\langle\nabla r_t(x_{i,t}),x_{i,t}-z_{i,t+1}\rangle
+\langle\nabla r_t(z_{i,t+1}),z_{i,t+1}-y\rangle\nonumber\\
&=\langle\nabla f_{i,t}(x_{i,t})+\nabla r_t(x_{i,t}),x_{i,t}-z_{i,t+1}\rangle
+\langle\nabla f_{i,t}(x_{i,t})+\nabla r_t(z_{i,t+1}),z_{i,t+1}-y\rangle\nonumber\\
&\le G_1\|\epsilon^z_{i,t}\|+\langle\nabla f_{i,t}(x_{i,t})+\nabla r_t(z_{i,t+1}),z_{i,t+1}-y\rangle.
\end{align}
For the second term on the right-hand side of \eqref{online_op:fxy}, we have
\begin{align}
&\langle\nabla f_{i,t}(x_{i,t})+\nabla r_t(z_{i,t+1}),z_{i,t+1}-y\rangle\nonumber\\
&=\langle(\nabla g_{i,t}(x_{i,t}))^\top q_{i,t+1},y-z_{i,t+1}\rangle
+\langle\omega_{i,t+1}+\nabla r_t(z_{i,t+1}),z_{i,t+1}-y\rangle\nonumber\\
&=\langle(\nabla g_{i,t}(x_{i,t}))^\top q_{i,t+1},y-x_{i,t}\rangle
+\langle(\nabla g_{i,t}(x_{i,t}))^\top q_{i,t+1},x_{i,t}-z_{i,t+1}\rangle\nonumber\\
&\quad+\langle\omega_{i,t+1}+\nabla r_t(z_{i,t+1}),z_{i,t+1}-y\rangle.\label{online_op:fxy1}
\end{align}

We next to find the upper bound of each term on the right-hand side of \eqref{online_op:fxy1}.

From (\ref{online_op:subgradient3}) and $q_{i,t}\ge{\bm 0}_{m_i},~\forall i\in[n],~\forall t\in\mathbb{N}_+$, we have
\begin{align}
&\langle(\nabla g_{i,t}(x_{i,t}))^\top q_{i,t+1},y-x_{i,t}\rangle
\le q_{i,t+1}^\top g_{i,t}(y) -q_{i,t+1}^\top g_{i,t}(x_{i,t}).\label{online_op:gyxdelta}
\end{align}
From the Cauchy--Schwarz inequality and \eqref{online_op:subgupper}, we have
\begin{align}
\langle(\nabla g_{i,t}(x_{i,t}))^\top q_{i,t+1},x_{i,t}-z_{i,t+1}\rangle
\le G_2\|q_{i,t+1}\|\|\epsilon^z_{i,t}\|.\label{online_op:gxxtilde}
\end{align}
Applying (\ref{online_op:lemma_mirroreq2}) to the update (\ref{online_op:al_x}), we get
\begin{align}
&\langle\omega_{i,t+1}+\nabla r_t(z_{i,t+1}),z_{i,t+1}-y\rangle\nonumber\\
&\le\frac{1}{\alpha_{t}}\Big(\calD_\psi(y,x_{i,t})-\calD_\psi(y,z_{i,t+1})-\calD_\psi(z_{i,t+1},x_{i,t})\Big)\nonumber\\
&=\frac{1}{\alpha_{t}}\Big(\calD_\psi(y,x_{i,t})-\calD_\psi(y,x_{i,t+1})
+\calD_\psi(y,x_{i,t+1})-\calD_\psi(y,z_{i,t+1})-\calD_\psi(z_{i,t+1},x_{i,t})\Big)\nonumber\\
&=\Delta_{i,t}(y)+\frac{1}{\alpha_{t}}\Big(
\calD_\psi\Big(y,\sum_{j=1}^n[W_{t+1}]_{ij}z_{j,t+1}\Big)
-\calD_\psi(y,z_{i,t+1})-\calD_\psi(z_{i,t+1},x_{i,t})\Big)\nonumber\\
&\le \Delta_{i,t}(y)+\frac{1}{\alpha_{t}}\Big(
\sum_{j=1}^n[W_{t+1}]_{ij}\calD_\psi(y,z_{j,t+1})
-\calD_\psi(y,z_{i,t+1})-\frac{\sigma}{2}\|z_{i,t+1}-x_{i,t}\|^2\Big),\label{online_op:omgea2}
\end{align}
where the second equality holds due to (\ref{online_op:al_z}); and the last inequality holds since $W_{t+1}$ is doubly stochastic, Assumption~\ref{online_op:assbregman}, and \eqref{online_op:eqbergman}.

Summing \eqref{online_op:fxy}--(\ref{online_op:omgea2}) over $i\in[n]$, dividing by $n$, using $\sum_{i=1}^n[W_t]_{ij}=1,~\forall t\in\mathbb{N}_+$, and rearranging terms yields (\ref{online_op:lemma_regretdeltaequ}).
\end{proof}

\subsection{Proof of Lemma~\ref{online_op:theoremreg-local}}\label{online_op:theoremregproof-local}

\noindent {\bf (i)}  Noting that $g_{i,t}(y)\le{\bm 0}_{m_i},~\forall i\in[n],~\forall t\in\mathbb{N}_+$ when $y\in\calX_{T}$, summing (\ref{online_op:lemma_regretdeltaequ}) over $t\in[T]$ gives
\begin{align}\label{online_op:theoremregequ2}
&\frac{1}{n}\sum_{i=1}^{n}\sum_{t=1}^{T}(l_{i,t}(x_{i,t})-l_{i,t}(y_{t}))
\le \frac{1}{n}\sum_{i=1}^n\sum_{t=1}^T\Big(-q_{i,t+1}^\top g_{i,t}(x_{i,t})  +\frac{1}{n}\tilde{\Delta}_t+\Delta_{i,t}(y)\Big).
\end{align}

We have
\begin{align}\label{online_op:xitxtildemu}
&\sum_{t=1}^{T}\sum_{i=1}^n(G_1+G_2\|q_{i,t+1}\|)\|\epsilon^z_{i,t}\|\le\sum_{t=1}^{T}\sum_{i=1}^n\Big(\frac{2G_1^2\alpha_{t}}{\sigma}+\frac{2G_2^2\alpha_{t}\|q_{i,t+1}\|^2}{\sigma}
	+\frac{\sigma\|\epsilon^z_{i,t}\|^2}{4\alpha_{t}}\Big).
\end{align}

From \eqref{online_op:al_q}, for all $t\in\mathbb{N}_+$, we have
\begin{align}
&\|q_{i,t+1}\|=\gamma_{t}\|[g_{i,t}(x_{i,t})]_+\|.\label{online_op:qibound}
\end{align}
Then, from \eqref{online_op:qibound_ab}, \eqref{online_op:qibound}, and $\gamma_t\alpha_t=\gamma_0$, we have
\begin{align}\label{online_op:gc}
&\frac{2G_2^2\alpha_{t}\|q_{i,t+1}\|^2}{\sigma}-q_{i,t+1}^\top g_{i,t}(x_{i,t})
=\Big(\frac{2G_2^2\gamma_0}{\sigma}-1\Big)\gamma_{t}\|[g_{i,t}(x_{i,t})]_+\|^2\le0,
\end{align}
where the inequality holds since $\gamma_0\le\sigma/(4G_2^2)$.

Combining (\ref{online_op:theoremregequ2})--\eqref{online_op:xitxtildemu} and \eqref{online_op:gc} yields \eqref{online_op:theoremregequ-local}.

\noindent {\bf (ii)}
From Assumption \ref{online_op:ass_ftgtupper}, we have
\begin{align}\label{online_op:ff}
	l_{i,t}(y)-l_{i,t}(x_{i,t})\le F,~\forall  y\in\mathbb{X}.
\end{align}

Dividing (\ref{online_op:lemma_regretdeltaequ}) by $\gamma_{t}$, using \eqref{online_op:ff}, and summing over $t\in[T]$ gives
\begin{align}
&\sum_{t=1}^T\sum_{i=1}^n\frac{q_{i,t+1}^\top g_{i,t}(x_{i,t})}{\gamma_t}\le  h_T+\sum_{t=1}^T\frac{nF}{\gamma_{t}}+\sum_{t=1}^T\frac{\tilde{\Delta}_t}{\gamma_{t}}
+\sum_{t=1}^T\sum_{i=1}^n\frac{\Delta_{i,t}(y)}{\gamma_{t}}
.\label{online_op:theoremregequ2_gslater}
\end{align}

We have
\begin{align}\label{online_op:xitxtildemu_g}
&\sum_{t=1}^{T}\sum_{i=1}^n\frac{(G_1+G_2\|q_{i,t+1}\|)\|\epsilon^z_{i,t}\|}{\gamma_{t}}\le \sum_{t=1}^{T}\sum_{i=1}^n\Big(\frac{2\gamma_0(G_1^2+G_2^2\|q_{i,t+1}\|^2)}{\sigma\gamma_{t}^2}+\frac{\sigma\|\epsilon^z_{i,t}\|^2}{4\gamma_0}\Big).
\end{align}

From $\gamma_t\alpha_t=\gamma_0$, we have
\begin{align}
\sum_{t=1}^T\frac{\Delta_{i,t}(y)}{\gamma_{t}}
&=\sum_{t=1}^T\frac{1}{\gamma_0}\Big(\calD_\psi(y,x_{i,t})-\calD_\psi(y,x_{i,t+1})\Big)\le\frac{\calD_\psi(y,x_{i,1})}{\gamma_0}.\label{online_op:dyz_g}
\end{align}

Combining \eqref{online_op:theoremregequ2_gslater}--\eqref{online_op:dyz_g}, \eqref{online_op:qibound}, and \eqref{online_op:qibound_ab}, and using $1/2\ge 2\gamma_0G_2^2/\sigma$ yields  \eqref{online_op:theoremregequ2_g2}.

\noindent {\bf (iii)} From (\ref{online_op:al_z}) and $\sum_{i=1}^n[W_t]_{ij}=\sum_{j=1}^n[W_t]_{ij}=1$,  we have
\begin{align}
\frac{1}{n}\sum_{i=1}^n\sum_{j=1}^n\|x_{i,t}-x_{j,t}\|
&=\frac{1}{n}\sum_{i=1}^n\sum_{j=1}^n\Big\|\sum_{l=1}^n[W_t]_{il}z_{l,t}-\bar{z}_{t}+\bar{z}_{t}-\sum_{l=1}^n[W_t]_{j,l}z_{l,t}\Big\|\nonumber\\
&\le2\sum_{i=1}^n\Big\|\sum_{j=1}^n[W_t]_{ij}z_{j,t}-\bar{z}_{t}\Big\|
=2\sum_{i=1}^n\Big\|\sum_{j=1}^n[W_t]_{ij}(z_{j,t}-\bar{z}_{t})\Big\|\nonumber\\
&\le2\sum_{i=1}^n\sum_{j=1}^n[W_t]_{ij}\|z_{j,t}-\bar{z}_{t}\|
=2\sum_{i=1}^n\|z_{i,t}-\bar{z}_t\|.\label{online_op:xz}
\end{align}
We have
\begin{align}
\sum_{t=3}^{T}\sum_{s=1}^{t-2}\lambda^{t-s-2}\sum_{j=1}^n\|\epsilon^z_{j,s}\|
&=\sum_{t=1}^{T-2}\sum_{j=1}^n\|\epsilon^z_{j,t}\|\sum_{s=0}^{T-t-2}\lambda^s
\le\frac{1}{1-\lambda}\sum_{t=1}^{T-2}\sum_{j=1}^n\|\epsilon^z_{j,t}\|.
\label{online_op:xitxbarT1}
\end{align}
From \eqref{online_op:xz}, \eqref{online_op:xitxbar}, and \eqref{online_op:xitxbarT1}, we have

\begin{align*}
&\frac{1}{n}\sum_{t=1}^{T}\sum_{i=1}^n\sum_{j=1}^n\|x_{i,t}-x_{j,t}\|
\le\sum_{t=1}^{T}\sum_{i=1}^n2\|z_{i,t}-\bar{z}_{t}\|\nonumber\\
&\le\sum_{t=1}^{T}\sum_{i=1}^n2\tau \lambda^{t-2}\sum_{j=1}^n\|z_{j,1}\|
+\sum_{t=2}^{T}\sum_{i=1}^n2\Big(\frac{1}{n}\sum_{j=1}^n\|\epsilon^z_{j,t-1}\|
+\|\epsilon^z_{i,t-1}\|\Big)+\sum_{t=3}^{T}\sum_{i=1}^n2\tau\sum_{s=1}^{t-2}\lambda^{t-s-2}\sum_{j=1}^n\|\epsilon^z_{j,s}\|\nonumber\\
&\le n\varepsilon_1+\sum_{t=2}^{T}\sum_{i=1}^n4\|\epsilon^z_{i,t-1}\|+\frac{2n\tau}{1-\lambda}\sum_{t=1}^{T-2}\sum_{j=1}^n\|\epsilon^z_{j,t}\|,
\end{align*}
which yields \eqref{online_op:xitxbarT2_slater}.

\noindent {\bf (iv)}
Similar to the way to get \eqref{online_op:xz}, from (\ref{online_op:al_z}), $\sum_{i=1}^n[W_t]_{ij}=\sum_{j=1}^n[W_t]_{ij}=1$, and $\|\cdot\|^2$ is convex, we have
\begin{align}
\frac{1}{n}\sum_{i=1}^n\sum_{j=1}^n\|x_{i,t}-x_{j,t}\|^2\le\sum_{i=1}^n4\|z_{i,t}-\bar{z}_t\|^2.\label{online_op:xz-squar}
\end{align}

From \eqref{online_op:xitxbar}, we have
\begin{align}
&4\sum_{t=1}^T\sum_{i=1}^n\|z_{i,t}-\bar{z}_{t}\|^2\nonumber\\
&\le 4\sum_{i=1}^n\sum_{t=1}^T\Big(\tau \lambda^{t-2}\sum_{j=1}^n\|z_{j,1}\|+\|\epsilon^z_{i,t-1}\|
+\frac{1}{n}\sum_{j=1}^n\|\epsilon^z_{j,t-1}\|+\tau\sum_{s=1}^{t-2}\lambda^{t-s-2}\sum_{j=1}^n\|\epsilon^z_{j,s}\|\Big)^2\nonumber\\
&\le 16\sum_{i=1}^n\sum_{t=1}^T\Big(\Big(\tau \lambda^{t-2}\sum_{j=1}^n\|z_{j,1}\|\Big)^2+\|\epsilon^z_{i,t-1}\|^2
+\Big(\frac{1}{n}\sum_{j=1}^n\|\epsilon^z_{j,t-1}\|\Big)^2
+\Big(\tau\sum_{s=1}^{t-2}\lambda^{t-s-2}\sum_{j=1}^n\|\epsilon^z_{j,s}\|\Big)^2\Big)\nonumber\\
&\le 16\sum_{i=1}^n\sum_{t=1}^T\Big(\Big(\tau \lambda^{t-2}\sum_{j=1}^n\|z_{j,1}\|\Big)^2+\|\epsilon^z_{i,t-1}\|^2
+\frac{1}{n}\sum_{j=1}^n\|\epsilon^z_{j,t-1}\|^2
+\tau^2\sum_{s=1}^{t-2}\lambda^{t-s-2}\sum_{s=1}^{t-2}\lambda^{t-s-2}\Big(\sum_{j=1}^n\|\epsilon^z_{j,s}\|\Big)^2\Big)\nonumber\\
&\le 16\sum_{i=1}^n\sum_{t=1}^T\Big(\Big(\tau \lambda^{t-2}\sum_{j=1}^n\|z_{j,1}\|\Big)^2+2\|\epsilon^z_{i,t-1}\|^2
+\frac{n\tau^2}{1-\lambda}\sum_{s=1}^{t-2}\lambda^{t-s-2}
\sum_{j=1}^n\|\epsilon^z_{j,s}\|^2\Big)\nonumber\\
&= 16\sum_{i=1}^n\sum_{t=1}^T\Big(\Big(\tau \lambda^{t-2}\sum_{j=1}^n\|z_{j,1}\|\Big)^2+2\|\epsilon^z_{i,t-1}\|^2\Big)+\frac{16n^2\tau^2}{1-\lambda}\sum_{j=1}^n\sum_{t=1}^{T-2}\|\epsilon^z_{j,t}\|^2\sum_{s=0}^{T-t-2}\lambda^{s}\nonumber\\
&\le \tilde{\varepsilon}_3+\tilde{\varepsilon}_4\sum_{t=1}^T\sum_{i=1}^n\|\epsilon^z_{i,t}\|^2,\label{online_op:xitxbarg}
\end{align}
where the third inequality holds due to the H\"{o}lder's inequality.

Thus, from \eqref{online_op:xz-squar}--\eqref{online_op:xitxbarg}, we know \eqref{online_op:xitxbargsquar} holds.

\noindent {\bf (v)} Applying (\ref{online_op:lemma_mirrorine}) to the update (\ref{online_op:al_x}) gives
\begin{align*}
\|\epsilon^z_{i,t}\|&=\|z_{i,t+1}-x_{i,t}\|\le\frac{\alpha_{t}\|\omega_{i,t+1}+\nabla r_{t}(x_{i,t})\|}{\sigma}\nonumber\\
&=\frac{\alpha_{t}\|\nabla f_{i,t}(x_{i,t})+(\nabla g_{i,t}(x_{i,t}))^\top q_{i,t+1}+\nabla r_{t}(x_{i,t})\|}{\sigma}\nonumber\\
&\le \frac{\alpha_{t}}{\sigma }(G_2\gamma_{t}\|[g_{i,t}(x_{i,t})]_+\|+G_1)=\frac{1}{\sigma}(G_2\gamma_{0}\|[g_{i,t}(x_{i,t})]_+\|+G_1\alpha_{t}),
\end{align*}
where the second equality holds due to (\ref{online_op:al_bigomega}); the last inequality holds due to (\ref{online_op:qibound}) and (\ref{online_op:subgupper}); and the last equality holds due to $\alpha_t\gamma_t=\gamma_0$. Thus, \eqref{online_op:xxt} holds.

\subsection{Proof of Lemma~\ref{online_op:theoremreg}}\label{online_op:theoremregproof}

\noindent {\bf (i)} From $l_t(x)=\frac{1}{n}\sum_{j=1}^nl_{j,t}(x)$, we have
\begin{align}\label{online_op:lxit}
\sum_{i=1}^{n}l_t(x_{i,t})&=\frac{1}{n}\sum_{i=1}^{n}\sum_{j=1}^{n}l_{j,t}(x_{i,t})=\frac{1}{n}\sum_{i=1}^{n}\sum_{j=1}^{n}l_{j,t}(x_{j,t})+\frac{1}{n}\sum_{i=1}^{n}\sum_{j=1}^{n}(l_{j,t}(x_{i,t})-l_{j,t}(x_{j,t}))\nonumber\\
&\le\sum_{i=1}^{n}l_{i,t}(x_{i,t})+\frac{1}{n}\sum_{i=1}^{n}\sum_{j=1}^{n}G_1\|x_{i,t}-x_{j,t}\|,
\end{align}
where the inequality holds due to \eqref{online_op:assfunction:functionLipf}.

From \eqref{online_op:xitxbarT2_slater}, we have
\begin{align}
\sum_{t=1}^{T}\sum_{i=1}^n\sum_{j=1}^nG_1\|x_{i,t}-x_{j,t}\|\le n\varepsilon_1G_1+\sum_{t=1}^{T}\sum_{i=1}^n\Big(\frac{\tilde{\varepsilon}_2^2G_1^2\alpha_t}{\sigma}+\frac{\sigma\|\epsilon^z_{i,t}\|^2}{4\alpha_t}\Big).\label{online_op:xitxbarT2}
\end{align}

Combining \eqref{online_op:lxit}--\eqref{online_op:xitxbarT2}  and \eqref{online_op:theoremregequ-local} yields
\eqref{online_op:theoremregequ}.

\noindent {\bf (ii)} We have
\begin{align}\label{online_op:gitsq}
\|[g_{i,t}(x_{i,t})]_+\|^2
&=\|[g_{i,t}(x_{i,t})]_+-[g_{i,t}(x_{j,t})]_++[g_{i,t}(x_{j,t})]_+\|^2\nonumber\\
&\ge\frac{1}{2}\|[g_{i,t}(x_{j,t})]_+\|^2-\|[g_{i,t}(x_{i,t})]_+-[g_{i,t}(x_{j,t})]_+\|^2\nonumber\\
&\ge\frac{1}{2}\|[g_{i,t}(x_{j,t})]_+\|^2-\|g_{i,t}(x_{i,t})-g_{i,t}(x_{j,t})\|^2\nonumber\\
&\ge\frac{1}{2}\|[g_{i,t}(x_{j,t})]_+\|^2-G_2^2\|x_{i,t}-x_{j,t}\|^2,
\end{align}
where the second  and the third inequalities hold due to the nonexpansive property of the projection $[\:\cdot\:]_+$ and \eqref{online_op:assfunction:functionLipg}, respectively.

From $g_t(x)=\col(g_{1,t}(x),\dots,g_{n,t}(x))$, we have
\begin{align}\label{online_op:gtsq}
	\sum_{t=1}^T\sum_{i=1}^n\sum_{j=1}^n\|[g_{i,t}(x_{j,t})]_+\|^2
	=\sum_{t=1}^T\sum_{j=1}^n\|[g_{t}(x_{j,t})]_+\|^2.
\end{align}

Summing (\ref{online_op:gitsq}) over $i,j\in[n],~t\in[T]$, dividing by $n$, and using \eqref{online_op:gtsq} and \eqref{online_op:xitxbargsquar} yields
\begin{align}
&\frac{1}{n}\sum_{t=1}^T\sum_{j=1}^n\|[g_{t}(x_{j,t})]_+\|^2\le\varepsilon_3+\sum_{t=1}^T\sum_{i=1}^n2(\|[g_{i,t}(x_{i,t})]_+\|^2+G_2^2\tilde{\varepsilon}_4\|\epsilon^z_{i,t}\|^2).
\label{online_op:theoremregequ2_g}
\end{align}

From $g_{i,t}(y)\le{\bm 0}_{m_i},~\forall i\in[n],~\forall t\in\mathbb{N}_+$ when $y\in\calX_{T}$, we have
\begin{align}\label{online_op:corollaryregequ10h}
h_T(y)\le0.
\end{align}

Combining \eqref{online_op:theoremregequ2_g}--\eqref{online_op:corollaryregequ10h} and \eqref{online_op:theoremregequ2_g2} yields
\begin{align}
\frac{1}{n}\sum_{i=1}^n\sum_{t=1}^T\|[g_{t}(x_{i,t})]_+\|^2\le\varepsilon_3+\varepsilon_4\tilde{h}_T,~\forall y\in\mathcal{X}_T.\label{online_op:theoremconsequ1}
\end{align}
Combining (\ref{online_op:theoremconsequ1}) and \eqref{online_op:constrain_vio_def} yields
(\ref{online_op:theoremconsequ}).

\noindent {\bf (iii)} From \eqref{online_op:xitxbarT2_slater} and \eqref{online_op:xxt}, we have
\begin{align}
\frac{1}{n}\sum_{t=1}^T\sum_{i=1}^n\sum_{j=1}^n\|x_{i,t}-x_{j,t}\|
&\le n\varepsilon_1+\tilde{\varepsilon}_2\sum_{t=1}^{T}\sum_{j=1}^n\|\epsilon^z_{j,t}\|\nonumber\\
&\le n\varepsilon_1+\frac{\tilde{\varepsilon}_2}{\sigma }\sum_{t=1}^{T}\sum_{j=1}^n(G_2\gamma_{0}\|[g_{j,t}(x_{j,t})]_+\|+G_1\alpha_{t}).
	\label{online_op:lemma_neterrorequ}
\end{align}

From \eqref{online_op:assfunction:functionLipg}, we have
\begin{align}
\frac{1}{n}\sum_{j=1}^n\sum_{t=1}^T\|[g_{t}(x_{j,t})]_+\|&\le\frac{1}{n}\sum_{i=1}^n\sum_{j=1}^n\sum_{t=1}^T\|[g_{i,t}(x_{j,t})]_+\|\nonumber\\
&=\frac{1}{n}\sum_{t=1}^T\sum_{i=1}^n\sum_{j=1}^n\|[g_{i,t}(x_{i,t})]_++[g_{i,t}(x_{j,t})]_+-[g_{i,t}(x_{i,t})]_+\|\nonumber\\
&\le\frac{1}{n}\sum_{t=1}^T\sum_{i=1}^n\sum_{j=1}^n(\|[g_{i,t}(x_{i,t})]_+\|+G_2\|x_{i,t}-x_{j,t}\|),\label{online_op:lemma-g-slater1}
\end{align}
Combining \eqref{online_op:lemma_neterrorequ} and \eqref{online_op:lemma-g-slater1} yields \eqref{online_op:lemma-g-slater}.

\subsection{Proof of Theorem~\ref{online_op:corollaryreg}}\label{online_op:corollaryregproof}
\noindent {\bf (i)}
From \eqref{online_op:stepsize1}, we have
\begin{align}
&\sum_{t=1}^T\alpha_{t}=\sum_{t=2}^{T}\frac{1}{t^c}+1
\le\int_{1}^{T}\frac{1}{t^c}dt+1\le\frac{T^{1-c}}{1-c}.
\label{online_op:corollaryregequ11}
\end{align}

Let $\alpha_0=\alpha_1$. From that $\{\alpha_t\}$ is nonincreasing and (\ref{online_op:bregmanupp}), we have
\begin{align}
\frac{1}{ n}\sum_{t=1}^T\sum_{i=1}^n\Delta_{i,t}(y)
&=\frac{1}{ n}\sum_{i=1}^n\sum_{t=1}^T\Big(\frac{1}{\alpha_{t-1}}\calD_\psi(y,x_{i,t})-\frac{1}{\alpha_{t}}\calD_\psi(y,x_{i,t+1})+\Big(\frac{1}{\alpha_{t}}-\frac{1}{\alpha_{t-1}}\Big)\calD_\psi(y,x_{i,t})\Big)\nonumber\\
&\le\frac{1}{ n}\sum_{i=1}^n\Big(\frac{1}{\alpha_{0}}\calD_\psi(y,x_{i,1})
-\frac{1}{\alpha_{T}}\calD_\psi(y,x_{i,T+1})\Big)+\Big(\frac{1}{\alpha_{T}}-\frac{1}{\alpha_{0}}\Big)K\nonumber\\
&\le\frac{K}{\alpha_{T}},~\forall y\in\mathbb{X}.\label{online_op:dyz}
\end{align}

Denote $$x^*=\argmin_{x\in\mathcal{X}_{T}}\sum_{t=1}^{T}l_t(x).$$
Choosing $y=x^*$, and combining  (\ref{online_op:theoremregequ}) and (\ref{online_op:corollaryregequ11})--\eqref{online_op:dyz} yields
\begin{align}\label{online_op:corollaryregequ1_proof}
\NetReg(T)\le\varepsilon_1G_1+\frac{\varepsilon_2}{1-c}T^{1-c}+KT^c,
\end{align}
which gives \eqref{online_op:corollaryregequ1}.

\noindent {\bf (ii)}
From \eqref{online_op:stepsize1}, we have
\begin{align}\label{online_op:corollaryregequ12}
\sum_{t=1}^T\frac{\gamma_0}{\gamma_{t}^2}\le\sum_{t=1}^T\frac{1}{\gamma_{t}}
=\frac{1}{\gamma_{0}}\sum_{t=1}^T\frac{1}{t^c}\le\frac{1}{\gamma_{0}(1-c)}T^{1-c}.
\end{align}

From \eqref{online_op:bregmanupp}, we have
\begin{align}
\sum_{i=1}^{n}\frac{\calD_\psi(y,x_{i,1})}{\gamma_0}\le\frac{nK}{\gamma_0},~\forall y\in\mathbb{X}.\label{online_op:dyz_g2}
\end{align}

Combining (\ref{online_op:theoremconsequ}) and  (\ref{online_op:corollaryregequ12})--\eqref{online_op:dyz_g2} yields
\begin{align}\label{online_op:corollaryconsequ_proof}
&\Big(\frac{1}{n}\sum_{i=1}^n\sum_{t=1}^T\|[g_{t}(x_{i,t})]_+\|\Big)^2
\le\varepsilon_3T+\frac{n\varepsilon_4KT}{\gamma_0}+\frac{n\varepsilon_4(F\sigma+2G_1^2)}{\sigma\gamma_{0}(1-c)}T^{2-c}.
\end{align}
Thus, \eqref{online_op:corollaryconsequ} holds.

\subsection{Proof of Theorem~\ref{online_op:corollaryreg_slater}}\label{online_op:corollaryregproof_slater}
\noindent {\bf (i)}  From \eqref{online_op:corollaryregequ1_proof}, we have \eqref{online_op:corollaryregequ1_slater}.

\noindent {\bf (ii)} Choosing $y=x_s$ in \eqref{online_op:theoremregequ2_g2}, and using \eqref{online_op:gtcon_hT} and \eqref{online_op:al_q} yields
\begin{align}
\epsilon_s\sum_{t=1}^T\sum_{i=1}^n\|[g_{i,t}(x_{i,t})]_+\|\le \tilde{h}_T(x_s).
\label{online_op:theoremregequ2_g2slater}
\end{align}
Combining \eqref{online_op:theoremregequ2_g2slater},  \eqref{online_op:stepsize1}, and \eqref{online_op:corollaryregequ12}--\eqref{online_op:dyz_g2} yields
\begin{align}
\sum_{i=1}^n\sum_{t=1}^T\|[g_{i,t}(x_{i,t})]_+\|
\le n\varepsilon_{7}T^{1-c},\label{online_op:theoremregequ2_g2slater2}
\end{align}
where $$\varepsilon_{7}=\frac{1}{\epsilon_s}\Big(\frac{F\sigma+2G_1^2}{\sigma\gamma_0(1-c)}+\frac{K}{\gamma_0}\Big).$$

From  \eqref{online_op:lemma-g-slater}, \eqref{online_op:stepsize1}, (\ref{online_op:corollaryregequ11}), and \eqref{online_op:theoremregequ2_g2slater2}, we have
\begin{align}
\frac{1}{n}\sum_{j=1}^n\sum_{t=1}^T\|[g_{t}(x_{j,t})]_+\|\le n\varepsilon_1G_2+\Big(\frac{\varepsilon_{5}}{1-c}+n\varepsilon_{6}\varepsilon_{7}\Big)T^{1-c},\label{online_op:theoremregequ2_g2slater3}
\end{align}
which yields \eqref{online_op:corollaryconsequ_slater}.

\subsection{Proof of Theorem~\ref{online_op:corollaryreg_sc}}\label{online_op:corollaryregproof_sc}

Since Assumption~\ref{online_op:assstrongconvex} holds, \eqref{online_op:fxy} can be replaced by
\begin{align}\label{online_op:fxy_sc}
l_{i,t}(x_{i,t})-l_{i,t}(y)\le G_1\|\epsilon^z_{i,t}\|-\mu\calD_{\psi}(y,x_{i,t})+\langle\nabla f_{i,t}(x_{i,t})+\nabla r_t(z_{i,t+1}),z_{i,t+1}-y\rangle.
\end{align}
Note that compared with \eqref{online_op:fxy}, \eqref{online_op:fxy_sc} has an extra term $-\mu\calD_{\psi}(y,x_{i,t})$.
Then, \eqref{online_op:theoremregequ} can be replaced by
\begin{align}
&\frac{1}{n}\sum_{i=1}^{n}\sum_{t=1}^{T}l_t(x_{i,t})-\sum_{t=1}^{T}l_t(y)\nonumber\\
&\le\varepsilon_1G_1+\varepsilon_2\sum_{t=1}^T\alpha_{t}+\frac{1}{n}\sum_{i=1}^{n}\sum_{t=1}^T(\Delta_{i,t}(y)-\mu\calD_{\psi}(y,x_{i,t})),~\forall y\in\mathcal{X}_T.
\label{online_op:theoremregequ_sc}
\end{align}
Moreover, \eqref{online_op:theoremregequ2_g2} can be replaced by
\begin{align}
&\sum_{t=1}^T\sum_{i=1}^n\frac{1}{2}\Big(\frac{q_{i,t+1}^\top g_{i,t}(x_{i,t})}{\gamma_t}+\frac{\sigma\|\epsilon^z_{i,t}\|^2}{2\gamma_0}\Big)\le h_T(y)+\tilde{h}_T(y)+\hat{h}_T(y),~\forall y\in\mathcal{X}_T,\label{online_op:theoremregequ2_g2_sc}
\end{align}
where $$\hat{h}_T(y)=-\sum_{i=1}^{n}\sum_{t=1}^T\frac{\mu\calD_{\psi}(y,x_{i,t})}{\gamma_t}.$$
As a result, \eqref{online_op:theoremconsequ} can be replaced by
\begin{align}
&\frac{1}{n}\sum_{i=1}^n\sum_{t=1}^T\|[g_{t}(x_{i,t})]_+\|
\le\sqrt{\varepsilon_3T+\varepsilon_4T(\tilde{h}_T(y)+\hat{h}_T(y))},~\forall y\in\mathcal{X}_T.\label{online_op:theoremconsequ_sc}
\end{align}

\noindent {\bf (i)}  We have
\begin{align}
&\frac{1}{n}\sum_{t=1}^T\sum_{i=1}^n(\Delta_{i,t}(y)-\mu\calD_{\psi}(y,x_{i,t}))\nonumber\\
&=\frac{1}{n}\sum_{i=1}^{n}\sum_{t=1}^T\Big(\frac{1}{\alpha_{t-1}}\calD_\psi(y,x_{i,t})-\frac{1}{\alpha_{t}}\calD_\psi(y,x_{i,t+1})
+\Big(\frac{1}{\alpha_{t}}-\frac{1}{\alpha_{t-1}}\Big)\calD_\psi(y,x_{i,t})
-\mu\calD_{\psi}(y,x_{i,t})\Big)\nonumber\\
&=\frac{1}{n}\sum_{i=1}^{n}\Big(\frac{1}{\alpha_{0}}\calD_\psi(y,x_{i,1})
-\frac{1}{\alpha_{T}}\calD_\psi(y,x_{i,T+1})
+\sum_{t=1}^T\Big(\frac{1}{\alpha_{t}}-\frac{1}{\alpha_{t-1}}-\mu\Big)
\calD_{\psi}(y,x_{i,t})\Big).\label{online_op:dyz_sc}
\end{align}

Denote $$\varepsilon_{8}=\Big\lceil\Big(\frac{1}{\mu}\Big)^{\frac{1}{1-c}}\Big\rceil.$$ From \eqref{online_op:stepsize1}, we have
\begin{align}\label{online_op:mu_sc}
\frac{1}{\alpha_{t+1}}-\frac{1}{\alpha_{t}}-\mu
=\frac{t+1}{(t+1)^{1-c}}-\frac{t}{t^{1-c}}-\mu
<\frac{1}{t^{1-c}}-\mu\le0,~\forall t\ge\varepsilon_{8}.
\end{align}

Choosing $y=x^*\in\calX_{T}$, and combining \eqref{online_op:bregmanupp}, (\ref{online_op:corollaryregequ11}), and \eqref{online_op:theoremregequ_sc}--\eqref{online_op:mu_sc} yields
\begin{align}\label{online_op:corollaryregequ1_sc_proof}
&\NetReg(T)\nonumber\\
&\le \varepsilon_1G_1+\frac{\varepsilon_2}{1-c}T^{1-c}+\frac{1}{n\alpha_{0}}\sum_{i=1}^{n}\calD_\psi(x^*,x_{i,1})+\frac{1}{n}\sum_{i=1}^{n}\sum_{t=1}^{\varepsilon_{8}}\Big(\frac{1}{\alpha_{t}}-\frac{1}{\alpha_{t-1}}-\mu\Big)\calD_{\psi}(x^*,x_{i,t})\nonumber\\
&\le \varepsilon_1G_1+\frac{\varepsilon_2}{1-c}T^{1-c}+\varepsilon_{8}[1-\mu]_+K.
\end{align}
Hence, \eqref{online_op:corollaryregequ1_sc} holds.

\noindent {\bf (ii)}  From \eqref{online_op:corollaryconsequ_proof}, we have \eqref{online_op:corollaryconsequ_sc}.

\noindent {\bf (iii)}  From \eqref{online_op:theoremregequ2_g2slater3}, we have \eqref{online_op:corollaryconsequ_slater_sc}.

\subsection{Proof of Theorem~\ref{online_op:corollaryreg_scmu}}\label{online_op:corollaryregproof_scmu}

We know that \eqref{online_op:theoremregequ_sc}--\eqref{online_op:theoremconsequ_sc} still hold.

\noindent {\bf (i)} From \eqref{online_op:stepsize1scmu} and \eqref{online_op:dyz_sc}, we have
\begin{align}
\frac{1}{n}\sum_{i=1}^n\sum_{t=1}^T(\Delta_{i,t}(y)-\mu\calD_{\psi}(y,x_{i,t}))\le0.\label{online_op:dyz_scmu}
\end{align}

From \eqref{online_op:stepsize1scmu}, we have
\begin{align}
&\sum_{t=1}^T\alpha_{t}=\sum_{t=1}^T\frac{1}{\mu t}
=\sum_{t=2}^{T}\frac{1}{\mu t}+\frac{1}{\mu }\le\int_{1}^{T}\frac{1}{\mu t}dt+\frac{1}{\mu }\le\frac{1}{\mu}(\log(T)+1).
\label{online_op:corollaryregequ11scmu}
\end{align}

Choosing $y=x^*\in\calX_{T}$ in \eqref{online_op:theoremregequ_sc}, and using \eqref{online_op:dyz_scmu}--(\ref{online_op:corollaryregequ11scmu}) yields
\begin{align*}
\NetReg(T)\le \varepsilon_1G_1+\frac{\varepsilon_2}{\mu}(\log(T)+1).
\end{align*}
Hence, \eqref{online_op:corollaryregequ1_scmu} holds.

\noindent {\bf (ii)}   From \eqref{online_op:stepsize1scmu}, we have

\begin{align*}
&\sum_{t=1}^T\frac{1}{\gamma_{t+1}}=\frac{1}{\gamma_{0}\mu}\sum_{t=1}^T\frac{1}{t}\le\frac{1}{\gamma_{0}\mu}(\log(T)+1),\\
&\sum_{t=1}^T\frac{\gamma_0}{\gamma_{t+1}^2}=\frac{1}{\gamma_{0}\mu^2}\sum_{t=1}^T\frac{1}{t^2}=\frac{1}{\gamma_{0}\mu^2}\Big(\sum_{t=2}^T\frac{1}{t^2}+1\Big)\le\frac{1}{\gamma_{0}\mu^2}\Big(\int_{t=2}^T\frac{1}{t^2}dt+1\Big)\le\frac{2}{\gamma_{0}\mu^2},\\
&\sum_{i=1}^{n}\frac{\calD_\psi(y,x_{i,1})}{\gamma_0}+\hat{h}_T(y)
\le\sum_{i=1}^{n}\frac{\calD_\psi(y,x_{i,1})}{\gamma_0}-\sum_{i=1}^{n}\frac{\mu\calD_{\psi}(y,x_{i,t})}{\gamma_1}=0,~\forall y\in\mathbb{X},
\end{align*}
which yield
\begin{align}
\tilde{h}_T(y)+\hat{h}_T(y)
\le\varepsilon_{9}+\frac{nF\log(T)}{\gamma_{0}\mu},
\label{online_op:theoremregequ2_g2slater2scmu}
\end{align}
where
$$
\varepsilon_{9}=\frac{nF}{\gamma_0\mu}
+\frac{4nG_1^2}{\sigma\gamma_{0}\mu^2}.$$

From \eqref{online_op:theoremconsequ_sc} and \eqref{online_op:theoremregequ2_g2slater2scmu}, we have
\begin{align*}
&\Big(\frac{1}{n}\sum_{i=1}^n\sum_{t=1}^T\|[g_{t}(x_{i,t})]_+\|\Big)^2
\le\varepsilon_3T+\varepsilon_4T\Big(\varepsilon_{9}+\frac{nF\log(T)}{\gamma_{0}\mu}\Big),
\end{align*}
which yields \eqref{online_op:corollaryconsequ_scmu}.

\noindent {\bf (iii)}  Choosing $y=x_s$ in \eqref{online_op:theoremregequ2_g2_sc}, and using \eqref{online_op:gtcon_hT} and \eqref{online_op:al_q} yields
\begin{align}
\epsilon_s\sum_{t=1}^T\sum_{i=1}^n\|[g_{i,t}(x_{i,t})]_+\|\le \tilde{h}_T(x_s)+\hat{h}_T(x_s).
\label{online_op:theoremregequ2_g2slater_sc}
\end{align}

From  \eqref{online_op:lemma-g-slater}, \eqref{online_op:stepsize1scmu}, and (\ref{online_op:corollaryregequ11scmu})--\eqref{online_op:theoremregequ2_g2slater_sc}, we have
\begin{align*}
&\frac{1}{n}\sum_{j=1}^n\sum_{t=1}^T\|[g_{t}(x_{j,t})]_+\|
\le n\varepsilon_1G_2+\frac{\varepsilon_5}{\mu}(\log(T)+1)+
\frac{\varepsilon_6}{\epsilon_s}\Big(\varepsilon_{9}+\frac{nF\log(T)}{\gamma_{0}\mu}\Big).
\end{align*}
Hence, \eqref{online_op:corollaryconsequ_slater_scmu} holds.

\bibliographystyle{IEEEtran}
\bibliography{refs}

\end{document}